\documentclass[reqno]{amsproc}
\usepackage{amssymb}
\usepackage{amsfonts}
\usepackage{amsmath}
\usepackage{graphicx}
\RequirePackage[cp1251]{inputenc}
\RequirePackage[english]{babel}
\RequirePackage{srcltx}

\theoremstyle{plain}

\newtheorem{corollary}{Corollary}[section]
\newtheorem*{corollary*}{Corollary}

\newtheorem{lemma}{Lemma}[section]

\newtheorem{theorem}{Theorem}[section]
\newtheorem*{thma}{Theorem A}

\theoremstyle{remark}
\newtheorem{remark}{Remark}
\newtheorem*{remark*}{Remark}
\newtheorem*{remarks*}{Remarks}

\theoremstyle{definition}

\newtheorem{example}{Example}
\newtheorem*{example*}{Example}

\numberwithin{equation}{section}
\newcommand\R{\mathbb{R}}
\newcommand\wh{\widehat}

\newcommand\ol{\overline}
\newcommand{\T}{\mathbb{{T}}}
\newcommand{\N}{\mathbb{{N}}}
\newcommand\Z{\mathbb{Z}}

\renewcommand\ll{\lesssim}
\renewcommand\gg{\gtrsim}

\newcommand\supp{\operatorname{supp}}
\renewcommand\Re{\operatorname{Re}}
\renewcommand\Im{\operatorname{Im}}

\begin{document}

\title[Moduli of smoothness and Fourier transforms]
{Moduli of smoothness and growth properties of Fourier transforms: two-sided estimates}
\author{D.~Gorbachev}
\address{D.~Gorbachev, Tula State University,
Department of Mechanics and Mathematics,
300600 Tula, Russia}
\email{dvgmail@mail.ru}

\author{S.~Tikhonov}
\address{S.~Tikhonov, ICREA and Centre de Recerca Matem\`{a}tica,
Apartat 50~08193 Bellaterra, Barce\-lona, Spain}
\email{stikhonov@crm.cat}

\thanks{This research was partially supported by the
MTM 2011-27637, RFFI 10-01-00564, RFFI 12-01-00170, 2009 SGR 1303.
The results of the paper were presented at the conference on Function spaces in CRM (Barcelona) in September 2011.
}

\date{March 9, 2010}
\subjclass[2000]{42B10, 26A15} \keywords{Fourier transforms, moduli of smoothness, Pitt's inequality}

\begin{abstract}
We prove two-sided inequalities between the integral moduli of smoothness of a
function on $\R^d/\T^d$ and the weighted tail-type integrals of its Fourier
transform/series.
 Sharpness of obtained results in particular is given by the
equivalence results for functions satisfying certain regular conditions.
Applications include a quantitative form of the Riemann–Lebesgue lemma as well
as several other questions in approximation theory and the theory of function
spaces.
\end{abstract}
\maketitle

\vspace{10mm}

\section{Introduction}
\label{sec-intr}

This paper studies the interrelation between the smoothness of a function and growth
properties of Fourier transforms/coefficients. Let us first recall the classical
Riemann--Lebesgue lemma: $|\wh{f}_n|\to 0$ as $|n|\to \infty$, where
$f\in L^1(\mathbb{T}^d)$. Its quantitative version, the Lebesgue type estimate
for the Fourier coefficients, is well known \cite[Vol.~I, Ch.~4, \S~4]{Zy} and
given by
\begin{equation}\label{fc}
|\wh{f}_n|\ll \omega_l\left(f, \frac{1}{|n|}\right)_1, \quad f\in L^1(\mathbb{T}^d),
\end{equation}
where the modulus of smoothness $\omega_l(f,\delta)_p$ of a function $f\in L^{p}(X)$
is defined by
\begin{equation}\label{mod}
\omega_l\left(f,\delta\right)_{p} = \sup_{|h|\le\delta}
\left\| \Delta^l_h f (x) \right\|_{L^p(X)},\quad 1\le p \le \infty,
\end{equation}
and
\[
\Delta^l_h f (x) = \Delta^{l-1}_h\left(\Delta_h f (x)\right),
\qquad \Delta_h f (x)=f(x+h)-f(x).
\]

For the Fourier transform, the estimate similar to \eqref{fc} can be found in, e.g., \cite{Tr}
\begin{equation}\label{fccc}
|\wh{f}(\xi)|\ll \omega_l\left(f,\frac{1}{|\xi|}\right)_1, \quad f\in L^1(\R^d),
\end{equation}
where the Fourier transform is given by
\begin{equation}\label{intF}
\wh f(\xi)=\int_{\R^d} f(x) e^{i\xi x}\,dx,\quad \xi\in \R^{d}.
\end{equation}
However, unlike \eqref{fc} the inequality \eqref{fccc} cannot be extended for the range $p>1$ (see section \ref{RLTR} below).

Very recently, Bray and Pinsky \cite{BP,BP2} and Ditzian \cite{Di} (see also
Gioev's paper \cite{G}) extended Lebesgue type estimate for the Fourier
transform/coefficients. We will need the following avarage function. For a
locally integrable function $f$ the average on a sphere in $\R^{d}$ of radius
$t>0$ is given by
\[
V_t f(x):=\frac{1}{m_t}\int_{|y-x|=t} f(y)\,dy\quad \text{with}\quad V_t
1=1,\quad d\ge 2.
\]
For $l\in \N$ we define
\[
V_{l,t}f (x):=\frac{-2}{\binom{2l}{l}} \sum_{j=1}^l (-1)^j
\binom{2l}{l-j} V_{jt} f(x).
\]

\begin{thma}
Let $f\in L^p(\R^d)$, $d\ge 2$, and $1\le p \le 2$, $1/p+1/p'=1$. Then
for $t>0$, $l\in \N$,
\begin{equation}\label{bray}
\left(
\int_{\R^d}\left[\min(1,t|\xi|)^{2l} |\wh{f}(\xi)|\right]^{p'}d\xi
\right)^{1/p'}\ll
\|f - V_{l,t} f\|_p, \quad 1< p \le 2,
\end{equation}
and
\begin{equation}\label{bray-in}
\sup_{\xi\in \R^{d}}\left[\min(1,t|\xi|)^{2l} |\wh{f}(\xi)|\right]\ll
\|f - V_{l,t} f\|_1.
\end{equation}
\end{thma}
Similar results were also proved for moduli of smoothness of functions on $\R$ and $\T^d$ (see \cite{Di}).
 In the rest of the paper we will assume that $t>0$, $l\in \N$, and
\begin{equation}\label{Ot-d2}
\Omega_{l}(f,t)_{p}=\|f - V_{l,t} f\|_p,\quad \theta=2,
\end{equation}
if $d\ge 2$ and
\begin{equation}\label{Ot-d1}
\Omega_{l}(f,t)_{p}=\omega_l(f,t)_p,\quad \theta=1
\end{equation}
if $d=1$.

 The main goal of this paper is to extend inequalities \eqref{bray} and \eqref{bray-in} in the following sense.
 First, we prove sharper estimates by considering the weighted $L^q$ norm of $\min(1,t|\xi|)^{\theta l} |\wh{f}(\xi)|$, that is,
\begin{equation}\label{mod-f101}
\Big\|\min(1,t|\xi|)^{\theta l} |\wh{f}(\xi)|\Big\|_{L^q(u)}\ll \Omega_{l}(f,t)_{p}, \qquad p\le q
  \end{equation}
with the certain weight function $u$. Then varying the parameter $q$ gives
us the better bound from below of $\Omega_{l}(f,t)_{p}$. In particular, if
$q=p'$ we arrive at \eqref{bray} and \eqref{bray-in}.

Second, we prove the reverse inequalities
showing how smoothness of a function depends on the average decay of its
Fourier transform:
\begin{equation}\label{mod-f102}
\Omega_{l}(f,t)_{p} \ll \Bigl\|\min(1,t|\xi|)^{\theta l} |\wh{f}(\xi)|\Bigr\|_{L^q(u)}, \qquad q\le p,
\end{equation}
Third, we define the class of general monotone functions and prove that for
this class the equivalence result holds:
\begin{equation}\label{mod-f103}
 \Omega_{l}(f,t)_{p} \asymp \Big\|\min(1,t|\xi|)^{\theta l} |\wh{f}(\xi)|\Big\|_{L^p(u)}.
\end{equation}
Note that for $p=2$, this follows from (\ref{mod-f101}) and (\ref{mod-f102}) in the general case (see also \cite{BP, G}).

The paper is organized as follows. In Section 2, we prove inequalities
\eqref{mod-f101} and \eqref{mod-f102} when $1< p\le 2$ and $p\ge 2$
respectively. In Section 3 we study inequalities \eqref{mod-f101} and
\eqref{mod-f102} in the case of radial functions and we show that, with a fixed
$p$, the range of the parameter $q$ is extended. In Section 4 we deal with the
general monotone functions. Again, we prove inequalities \eqref{mod-f101} and
\eqref{mod-f102} under wider range of the parameter $q$ than in the case of
radial functions. Moreover, we show equivalence \eqref{mod-f103} in this case.
Section~5 studies inequalities \eqref{mod-f101} and \eqref{mod-f102} for
functions on $\T^d$, $d\ge 1.$
In Section~6 we obtain the equivalence
result of type \eqref{mod-f103} for periodic functions
whose sequence of Fourier coefficients is general monotone. Section 7 considers several application of
obtained results in approximation theory (sharp relations between best
approximations and moduli of smoothness) and functional analysis (embedding
theorems, characterization of the Lipschitz/Besov spaces in terms of the
Fourier transforms).

Finally, we remark that inequalities between moduli of smoothness and the
Fourier transform in the Lebesgue and Lorentz spaces were studied earlier in
\cite{cline} and \cite{GK}.

\section{Growth of Fourier transforms via moduli of smoothness. The general case}
\label{sec-gc}

The following theorem is the main result of this section.

\begin{theorem}\label{t1}
Let $f\in L^{p}(\R^{d})$, $d\ge 1$.

\smallbreak
\noindent
\textup{(A)} Let $1<p\le 2$. Then for $p\le q\le p'$ we have
$|\xi|^{d(1-1/p-1/q)}\wh{f}(\xi)\in L^{q}(\R^{d})$, and
\begin{equation}\label{bray1}
\left(
\int_{\R^d}\left[\min(1,t|\xi|)^{\theta l} |\xi|^{d(1-1/p-1/q)} |\wh{f}(\xi)|\right]^{q}\,d\xi
\right)^{1/q}\ll
\Omega_l(f,t)_p.
\end{equation}

\smallbreak
\noindent
\textup{(B)} Let $2\le p<\infty$, $|\xi|^{d(1-1/p-1/q)}\wh{f}(\xi)\in
L^{q}(\R^{d})$, $q>1$, and $\max\left\{q,q'\right\}\le p$. Then
\begin{equation}\label{bray2}
\left(
\int_{\R^d}\left[\min(1,t|\xi|)^{\theta l} |\xi|^{d(1-1/p-1/q)} |\wh{f}(\xi)|\right]^{q}\,d\xi
\right)^{1/q}\gg
\Omega_l(f,t)_p.
\end{equation}
\end{theorem}

\begin{remark*}
Theorem~A follows from Theorem \ref{t1}~(A) (take $q=p'$). In part (B) we
assume that for $f\in L^{p}(\R^{d})$ the Fourier transform $\wh{f}$ is well
defined and such that
$|\xi|^{d(1-1/p-1/q)}\wh{f}(\xi)\in L^{q}(\R^{d})$ for a certain $q>1$
satisfying $\max\left\{q,q'\right\}\le p$.
\end{remark*}

\begin{proof}[Proof of Theorem \ref{t1}]
We will use the following Pitt's inequality \cite{BH} (see also \cite{GLT}):
\begin{equation}\label{PittIneq}
\left(\int_{\R^{d}}\left(|\xi|^{-\gamma}|\wh{g}(\xi)|\right)^{q}\,d\xi\right)^{1/q}\ll
\left(\int_{\R^{d}}\left(|x|^{\beta}|g(x)|\right)^{p}\,dx\right)^{1/p},
\end{equation}
where
\begin{equation}\label{PittIneq-cond}
\beta-\gamma=d\left(1-\frac{1}{p}-\frac{1}{q}\right),\quad
\max\left\{0,d\left(\frac{1}{p}+\frac{1}{q}-1\right)\right\}\le
\gamma<\frac{d}{q},\quad 1<p\le q<\infty.
\end{equation}
Here the Fourier transform $\wh{g}$ is understood in the usual sense of
weighted Fourier inequality \eqref{PittIneq}; see, e.g., \cite[Sect.~1, 2]{BL}.

Let us write inequality \eqref{PittIneq} with change of parameters $\wh{g}\leftrightarrow f$,
$p\leftrightarrow q$, $\beta\leftrightarrow -\gamma$. Let
$|\xi|^{-\gamma}\wh{f}(\xi)\in L^{q}(\R^{d})$, then
\begin{equation}\label{PittIneq2}
\left(\int_{\R^{d}}\left(|\xi|^{-\gamma}|\wh{f}(\xi)|\right)^{q}\,d\xi\right)^{1/q}\gg
\left(\int_{\R^{d}}\left(|x|^{\beta}|f(x)|\right)^{p}\,dx\right)^{1/p},
\end{equation}
where
\begin{equation}\label{PittIneq2-cond}
\beta-\gamma=d\left(1-\frac{1}{p}-\frac{1}{q}\right),\quad
\max\left\{0,d\left(\frac{1}{p}+\frac{1}{q}-1\right)\right\}\le
-\beta<\frac{d}{p},\quad 1<q\le p<\infty.
\end{equation}

\subsubsection*{The case of $d\ge 2$.} Then by \eqref{Ot-d2},
$\Omega_{l}(f,t)_{p}=\|f - V_{l,t} f\|_p$, $\theta=2$. Let us write the
left-hand side in \eqref{bray1} and \eqref{bray2} as
\[
I:=\left\|\min(1,t|\xi|)^{2l}h(\xi)\right\|_{q},\qquad
h(\xi)=|\xi|^{d(1-1/p-1/q)}|\wh{f}(\xi)|.
\]

In \cite[Cor.~2.3, Th.~3.1]{DD}, it is shown that for $f\in
L^p(\R^d)$, $1\le p\le \infty$, $t>0$, and integer~$l,$
\begin{equation}\label{v}
\|f - V_{l,t} f\|_p\asymp
K_l(f,\Delta, t^{2l})_p\asymp
R_l(f,\Delta, t^{2l})_p,
\end{equation}
where
\begin{equation*}
K_l(f,\Delta, t^{2l})_p:= \inf\big\{
\|f-g\|_p+t^{2l}\|\Delta^lg\|_p\colon \Delta^lg\in L^p(\R^d)
\big\},
\end{equation*}
the Laplacian is given by $\Delta=\frac{\partial^{2}}{\partial
x_1^{2}}+\dots+\frac{\partial^{2}}{\partial x_d^{2}}$,
\begin{equation}\label{v1-new}
\begin{gathered}
R_l(f,\Delta, t^{2l})_p:= \|f-R_{\lambda,l,b}(f)\|_p+t^{2l}\|\Delta^l
R_{\lambda,l,b}(f)\|_p,
\\
\lambda=1/t,\quad b\ge d+2.
\end{gathered}
\end{equation}
Here (see \cite[Sec. 2]{DD})
\[
R_{\lambda,l,b}(f)(x)=(G_{\lambda,l,b}*f)(x),\qquad
G_{\lambda,l,b}(x)=\lambda^{d}G_{l,b}(\lambda x),\quad
\wh{G_{l,b}}(\xi)=\eta_{l,b}(|\xi|),
\]
where
\begin{equation}\label{defeta}
\eta_{l,b}(s)=(1-s^{2l})_{+}^{b},\quad s=|\xi|\ge 0,
\end{equation}
and
\begin{equation}\label{hatR}
\begin{gathered}
\left[R_{\lambda,l,b}(f)\right]\!\wh{\;\;}(\xi)=\eta_{l,b}(t|\xi|)\wh{f}(\xi),\quad
\left[f-R_{\lambda,l,b}(f)\right]\!\wh{\;\;}(\xi)=\left[1-\eta_{l,b}(t|\xi|)\right]\!\wh{f}(\xi),
\\
\left[\Delta^l
R_{\lambda,l,b}(f)\right]\!\wh{\;\;}(\xi)
=
(-1)^{l}|\xi|^{2l}\left[R_{\lambda,l,b}(f)\right]\!\wh{\;\;}(\xi)=
(-1)^{l}|\xi|^{2l}\eta_{l,b}(t|\xi|)\wh{f}(\xi).
\end{gathered}
\end{equation}
Also,
\begin{equation}\label{nG1}
\|G_{\lambda,l,b}(x)\|_{1}=\|G_{l,b}\|_{1}<\infty.
\end{equation}

Taking into account that, for $b>0$,
\[
\eta_{l,b}(s)\sim 1-bs^{2l},\quad s\to 0,\qquad
\eta_{l,b}(s)=0,\quad s\ge 1,
\]
we obtain
\begin{equation}\label{etamin}
1-\eta_{l,b}(s)\asymp \min(1,s)^{2l},\quad s\ge 0.
\end{equation}
Changing variables $b \leftrightarrow b+1$ gives
\begin{equation*}
\min(1,s)^{2l}\asymp 1-\eta_{l,b+1}(s)=1-(1-s^{2l})\eta_{l,b}(s)=
1-\eta_{l,b}(s)+s^{2l}\eta_{l,b}(s).
\end{equation*}
Therefore,
\begin{equation}\label{I}
I=\left\|\min(1,t|\xi|)^{2l}h(\xi)\right\|_{q}\asymp
\left\|\left[1-\eta_{l,b}(t|\xi|)+(t|\xi|)^{2l}\eta_{l,b}(t|\xi|)\right]h(\xi)\right\|_{q}.
\end{equation}
Define
\begin{equation}\label{h1h2}
h_{1}(\xi)=\left[1-\eta_{l,b}(t|\xi|)\right]h(\xi),\quad
h_{2}(\xi)=(t|\xi|)^{2l}\eta_{l,b}(t|\xi|)h(\xi).
\end{equation}
Note that both $h_1$ and $h_2$ are non-negative. For non-negative functions we
have
\begin{equation}\label{h1h2eq}
\|h_{1}+h_{2}\|_{q}\asymp \|h_{1}\|_{q}+\|h_{2}\|_{q},\quad 1\le q\le \infty.
\end{equation}
This, \eqref{I}, and \eqref{h1h2} yield
\[
I\asymp
\left\||\xi|^{d(1-1/p-1/q)}\left[1-\eta_{l,b}(t|\xi|)\right]|\wh{f}(\xi)|\right\|_{q}+
\left\||\xi|^{d(1-1/p-1/q)}(t|\xi|)^{2l}\eta_{l,b}(t|\xi|)|\wh{f}(\xi)|\right\|_{q},
\]
or, by \eqref{hatR},
\begin{equation}\label{Iq}
I\asymp
\left\||\xi|^{d(1-1/p-1/q)}\left|\left[f-R_{\lambda,l,b}(f)\right]\!\wh{\;\;}(\xi)\right|\right\|_{q}+
t^{2l}\left\||\xi|^{d(1-1/p-1/q)}\left|\left[\Delta^l
R_{\lambda,l,b}(f)\right]\!\wh{\;\;}(\xi)\right|\right\|_{q}.
\end{equation}

Now to prove (A), we assume that $p\le q$ and we use \eqref{Iq} and Pitt's inequality \eqref{PittIneq} with $\beta=0$.
In this case
$\gamma=d\left(\frac{1}{p}+\frac{1}{q}-1\right)$ and $\gamma\ge 0$ (see \eqref{PittIneq-cond}).
The latter is ensured by
 $q\le p'$. Then $|\xi|^{d(1-1/p-1/q)}\wh{f}(\xi)\in L^{q}(\R^{d})$ and
\[
I\ll
\left\|f-R_{\lambda,l,b}(f)\right\|_{p}+t^{2l}\left\|\Delta^l
R_{\lambda,l,b}(f)\right\|_{p}.
\]
Combining this with \eqref{v}, and \eqref{v1-new} we get~(A).

In part (B) we assume that $q\le p$. Inequality \eqref{bray2} follows from
\eqref{Iq} and inequality \eqref{PittIneq2} for $\beta=0$. In this case, by
\eqref{PittIneq2-cond}, $\gamma=d\left(\frac{1}{p}+\frac{1}{q}-1\right)$ and
$\max\{0,\gamma\}\le 0$, i.e., $\gamma\le 0$. The latter is $q\ge p'$ or,
equivalently, $q'\le p$.

\subsubsection*{The case of $d=1$.} According to \eqref{Ot-d1}, we have
$\Omega_{l}(f,t)_{p}=\omega_l(f,t)_p$ and $\theta=1.$ The proof of key steps is
similar to the proof in the case of $d\ge 2$. The only difference is the
realization result (\cite{DIT}) given~by
\[
\omega_{l}(f,t)_{p}
\asymp
\inf\left(\|f-g\|_{p}+t^{l}\|g^{(l)}\|_{p}\colon
g^{(l)}\in \mathrm{E}_{\lambda}\cap L^{p}(\R)\right)\asymp
\|f-g_{\lambda}\|_{p}+t^{l}\|g_{\lambda}^{(l)}\|_{p},\quad \lambda=1/t,
\]
where $\mathrm{E}_{\lambda}$ is the collection of all entire functions of
exponential type~$\lambda$ and $g_{\lambda}\in \mathrm{E}_{\lambda}$ is such
that
\[
\|f-g_{\lambda}\|_p\ll E_\lambda(f)_p :=\inf_{g\in \mathrm{E}_{\lambda}}\|f-g\|_p.
\]
Since $\|g_{\lambda}^{(l)}\|_{p}\asymp \|Hg_{\lambda}^{(l)}\|_{p}$,
$1<p<\infty$, where $H$ is the Hilbert transform \cite[Ch.~5]{Tit}, then
$\omega_{l}(f,t)_{p} \asymp
\|f-g_{\lambda}\|_{p}+t^{l}\|D_{l}g_{\lambda}\|_{p}$, where $D_{l}=(id/dx)^{l}$
for even $l$ and $D_{l}=-iH(id/dx)^{l}$ for odd $l$.

Let $\chi_{\lambda}:=\chi_{[0,\lambda]}$. As Hille and Tamarkin \cite{HT}
showed, if $S_{\lambda}(f)$ is the partial Fourier integral of~$f$,~i.e.,
\begin{equation}\label{S1}
\left[S_{\lambda}(f)\right]\!\wh{\;\;}(\xi)=\chi_{\lambda}(|\xi|)\wh{f}(\xi),
\end{equation}
we have
\[
\|S_{\lambda}(f)\|_p\ll\|f\|_p, \quad 1<p<\infty.
\]
Then (see also \cite{Ti}) $g_{\lambda}$ can be taken as $S_{\lambda}(f)$, that
is, $\|f-S_{\lambda}(f)\|_p\ll E_\lambda(f)_p$. Therefore, for $1<p<\infty$,
\begin{equation}\label{omega2}
\omega_{l}(f,t)_{p}\asymp
\left\|f-S_{\lambda}(f)\right\|_{p}+t^{l}\bigl\|S_{\lambda}^{(l)}(f)\bigr\|_{p}
\asymp
\left\|f-S_{\lambda}(f)\right\|_{p}+t^{l}\left\|D_{l}S_{\lambda}(f)\right\|_{p},
\end{equation}
where
\begin{equation}\label{S2}
\bigl[S_{\lambda}^{(l)}(f)\bigr]\!\wh{\;\;}(\xi)=(-i\xi)^{l}\chi_{\lambda}(|\xi|)\wh{f}(\xi),\qquad
\left[D_{l}S_{\lambda}(f)\right]\!\wh{\;\;}(\xi)=|\xi|^{l}\chi_{\lambda}(|\xi|)\wh{f}(\xi).
\end{equation}

For $s\ge 0$ we have $\min(1,s)^{l}=1-\chi_{1}(s)+s^{l}\chi_{1}(s)$ and $\chi_{1}(ts)=\chi_{\lambda}(s)$, which gives
\begin{equation}\label{thm1-min}
\min(1,ts)^{l}=1-\chi_{\lambda}(s)+(ts)^{l}\chi_{\lambda}(s).
\end{equation}
This, \eqref{h1h2eq}, \eqref{S1}, and \eqref{S2} imply
\begin{multline}\label{IS}
I:=\left\|\min(1, t|\xi|)^{l} |\xi|^{1-1/p-1/q} |\wh{f}(\xi)|\right\|_{q}=
\left\|\left[1-\chi_{\lambda}(|\xi|)+(t|\xi|)^{l}\chi_{\lambda}(|\xi|)\right] |\xi|^{1-1/p-1/q}
|\wh{f}(\xi)|\right\|_{q}
\\
\asymp
\left\||\xi|^{1-1/p-1/q}[1-\chi_{\lambda}(|\xi|)]|\wh{f}(\xi)|\right\|_{q}+
\left\||\xi|^{1-1/p-1/q}(t|\xi|)^{l}\chi_{\lambda}(|\xi|)|\wh{f}(\xi)|\right\|_{q}
\\
=
\left\||\xi|^{1-1/p-1/q}
\left|\left[f-S_{\lambda}(f)\right]\!\wh{\;\;}(\xi)\right|\right\|_{q}+
t^{l}\left\||\xi|^{1-1/p-1/q}
\left|\left[D_{l}S_{\lambda}(f)\right]\!\wh{\;\;}(\xi)\right|\right\|_{q},
\end{multline}
which is an analogue of 
 \eqref{Iq}. Then as in the case of $d\ge 2$
we continue by using Pitt's inequality \eqref{PittIneq} and its corollary
\eqref{PittIneq2} with $\beta=0$ and $d=1$. This concludes the proof of the
case $d=1$.
\end{proof}

\section{Growth of Fourier transforms via moduli of smoothness. The case of
radial functions} \label{sec-rc}

Theorem \ref{t1} was proved under the condition $1<p\le q\le p'<\infty$ (A) and
$1<\max\left\{q,q'\right\}\le p<\infty$ (B). When $d\ge 2$ these conditions can
be extended if we restrict ourselves to radial functions
\[
f(x)=f_{0}(|x|).
\]

The Fourier transform of a radial function is also radial $\wh{f}(\xi)=F_{0}(|\xi|)$
(see \cite[Ch.~4]{SW}) and it can be written as the Fourier--Hankel transform
\[
F_{0}(s)=|S^{d-1}|\int_{0}^{\infty}f_{0}(t)j_{d/2-1}(st)t^{d-1}\,dt,
\]
where $j_{\alpha}(t)=\Gamma(\alpha+1)(t/2)^{-\alpha}J_{\alpha}(t)$ is the
normalized Bessel function ($j_{\alpha}(0)=1$), $\alpha\ge -1/2$. Useful
properties of $J_{\alpha}$ can be found in, e.g., \cite[Ch.~9]{AS}; see also
\cite{GLT} for some properties of $j_{\alpha}$.

\begin{theorem}\label{t2} Let $f\in L^p(\R^d)$ be a radial function and
$d\ge 2$.

\smallbreak
\noindent
\textup{(A)} Let $1<p\le q<\infty$. Then, for $p\le \frac{2d}{d+1}$, $q<\infty$
or $\frac{2d}{d+1}<p\le 2$, $p\le q\le
\left(\frac{d+1}{2}-\frac{d}{p}\right)^{-1}$,
\[
\left(
\int_{\R^d}\left[\min(1, t|\xi|)^{2l} |\xi|^{d(1-1/p-1/q)} |\wh{f}(\xi)|\right]^{q}\,d\xi
\right)^{1/q}\ll
\|f - V_{l,t} f\|_p.
\]

\smallbreak
\noindent
\textup{(B)} Let $2\le p<\infty$, $|\xi|^{d(1-1/p-1/q)}\wh{f}(\xi)\in
L^{q}(\R^{d})$, $q>1$ and
$\max\left\{q,d\left(\frac{d+1}{2}-\frac{1}{q}\right)^{-1}\right\}\le p$. Then
\[
\left(
\int_{\R^d}\left[\min(1, t|\xi|)^{2l} |\xi|^{d(1-1/p-1/q)} |\wh{f}(\xi)|\right]^{q}\,d\xi
\right)^{1/q}\gg
\|f - V_{l,t} f\|_p.
\]
\end{theorem}

\begin{remark*}
1. Formally, when $d=1$ conditions in Theorems \ref{t2} and \ref{t1} coincide.
However, note that no regularity condition was assumed in Theorem \ref{t1}.

2. The range of conditions on $p$ and $q$ in Theorem \ref{t2} is wider than the
corresponding range in Theorem \ref{t1} for $d\ge 2$.

Indeed, in Theorem \ref{t1}~(A) we assume the following conditions: $1<p\le 2$
and $p\le q\le p'$. If $p\le \frac{2d}{d+1}$, in Theorem \ref{t2}~(A)
conditions are $p\le q<\infty$. If $\frac{2d}{d+1}<p\le 2$, then
$\left(\frac{d+1}{2}-\frac{d}{p}\right)^{-1}\ge p'$. Thus, the conditions $p\le
q\le \left(\frac{d+1}{2}-\frac{d}{p}\right)^{-1}$ are less restrictive than
$p\le q\le p'$.

In its turn, in Theorem \ref{t1}~(B) we assume that $2\le p<\infty$ and
$\max\left\{q,q'\right\}\le p$. If $q<2$, then $p\ge q'$ and
$\max\left\{q,d\left(\frac{d+1}{2}-\frac{1}{q}\right)^{-1}\right\}=d\left(\frac{d+1}{2}-\frac{1}{q}\right)^{-1}<q'$.
If $2\le q$, then
$\max\left\{q,d\left(\frac{d+1}{2}-\frac{1}{q}\right)^{-1}\right\}=q$. Hence,
we get $\max\left\{q,d\left(\frac{d+1}{2}-\frac{1}{q}\right)^{-1}\right\}\le
\max\left\{q,q'\right\}$.
\end{remark*}

\begin{proof}[Proof of Theorem \ref{t2}] The proof is similar to the proof of
Theorem \ref{t1} but we use Pitt's inequality for radial functions. We also
remark that for a radial function $f$, functions $f-R_{\lambda,l,b}(f)$ and
$\Delta^{l}R_{\lambda,l,b}(f)$ are radial as well.

De~Carli \cite{Ca} proved Pitt's inequality for the Hankel transform. In
particular, this gives inequality \eqref{PittIneq} for radial functions. As it
was shown in \cite{Ca}, in this case the condition on $\gamma$ is as follows
\begin{equation}\label{radcond}
\frac{d}{q}-\frac{d+1}{2}+\max\left\{\frac{1}{p},\frac{1}{q'}\right\}\le
\gamma<\frac{d}{q},\quad 1<p\le q<\infty.
\end{equation}
Therefore, \eqref{PittIneq2} for radial functions holds under the condition
\begin{equation}\label{radcond2}
\frac{d}{p}-\frac{d+1}{2}+\max\left\{\frac{1}{q},\frac{1}{p'}\right\}\le
-\beta<\frac{d}{p},\quad 1<q\le p<\infty.
\end{equation}
We will use \eqref{radcond} and \eqref{radcond2} with $\beta=0$ and $\gamma=d\left(\frac{1}{p}+\frac{1}{q}-1\right)$.

To show (A), we assume \eqref{radcond}, that is, the following two conditions hold simultaneously
\[
\frac{d-1}{2}+\frac{1}{p}\le \frac{d}{p},\qquad
\frac{d-1}{2}+\frac{1}{q'}\le \frac{d}{p}.
\]
If $d\ge 2$, the first condition is equivalent to $p\le 2$. If $p\le
\frac{2d}{d+1}$, then the second condition is $q<\infty$. If
$\frac{2d}{d+1}<p\le 2$, then respectively $q\le \left(\frac{d+1}{2}-\frac{d}{p}\right)^{-1}$.

Let us verify all conditions in (B). We assume \eqref{radcond2}, or,
equivalently,
\[
\frac{d}{p}-\frac{d+1}{2}+\frac{1}{q}\le 0,\qquad
\frac{d}{p}-\frac{d+1}{2}+\frac{1}{p'}\le 0.
\]
If $d\ge 2$, the second inequality is equivalent to the condition $p\ge 2$. The
first inequality can be rewritten as $p\ge
d\left(\frac{d+1}{2}-\frac{1}{q}\right)^{-1}$. Since also $p\ge q$, we finally
arrive at condition
$\max\left\{q,d\left(\frac{d+1}{2}-\frac{1}{q}\right)^{-1}\right\}\le p$, under
which needed Pitt's inequality holds.
\end{proof}

\section{Growth of Fourier transforms via moduli of smoothness.\\ The case of general monotone functions}
\label{sec-gm}

 The following equivalence holds for
$p=2$ (see \cite{BP}, \cite{Di}, \cite{G} and Theorem \ref{t1} (A, B)):
\begin{equation}\label{eqL2}
\left(\int_{\R^d}\left[\min(1,t|\xi|)^{\theta l}|\wh{f}(\xi)|\right]^{p}\,d\xi
\right)^{1/p}\asymp \Omega_{l}(f,t)_{p},
\end{equation}
where $\Omega_{l}(f,t)_{p}$ and $\theta$ are given by \eqref{Ot-d2} and \eqref{Ot-d1}.

In this section we show that similar two sided inequalities also hold for
$\frac{2d}{d+1}<p<\infty$ provided $\wh{f}$ is radial, nonnegative and regular
in a certain sense.

\subsection{General monotone functions and the $\wh{GM}{}^{d}$ class}
A function $\varphi(z)$, $z>0$, is called \textit{general monotone}
($\varphi\in GM$), if it is locally of bounded variation on $(0,\infty)$,
vanishes at infinity, and for some constant $c>1$ depending on $\varphi$, the
following is true
\begin{equation}\label{GMcond-usi}
\int_{z}^{\infty}|d\varphi(u)|\ll
\int_{z/c}^{\infty}\frac{|\varphi(u)|}{u}\,du<\infty,\quad z>0
\end{equation}
(see \cite{GLT}). Any monotone function vanishing at infinity
satisfies GM-condition. Note also that \eqref{GMcond-usi} implies
\begin{equation}\label{estv}
|\varphi(z)|\ll \int_{z/c}^{\infty}\frac{|\varphi(u)|}{u}\,du.
\end{equation}
In particular, the latter gives, for any $b>1$,
\begin{equation}\label{GLT3-1}
|\varphi(z)|\ll \int_{z/(bc)}^{\infty}u^{-1}\left(\int_{u/b}^{bu}\frac{|\varphi(v)|}{v}\,dv\right)du.
\end{equation}

We will also use the following result on multipliers of general monotone functions.
\begin{lemma}\label{lemGMcond}
Let $\varphi\in GM$ and a function $\alpha(z)$ be locally of bounded variation
on $(0,\infty)$ such that $\lim_{z\to 0}\alpha(z)=0$ and
\[
\int_{0}^{cu}|d\alpha(v)|\ll |\alpha(u)|,\quad u>0.
\]
Then $\varphi_{1}=\alpha\varphi\in GM$.
\end{lemma}

\begin{proof}
By definition of GM, it is sufficient to verify
\begin{equation}\label{Iv1}
I:=\int_{z}^{\infty}|d\varphi_{1}(u)|\ll
\int_{z/c}^{\infty}\frac{|\varphi_{1}(u)|}{u}\,du,\quad z>0.
\end{equation}
First,
\[
I\ll
\int_{z}^{\infty}|\varphi(u)|\,|d\alpha(u)|+
\int_{z}^{\infty}|\alpha(u)|\,|d\varphi(u)|=:I_{1}+I_{2},
\]
and, by \eqref{estv}, we get
\[
I_{1}=\int_{z}^{\infty}|\varphi(u)|\,|d\alpha(u)|\ll
\int_{z}^{\infty}\left(\int_{u/c}^{\infty}\frac{|\varphi(v)|}{v}\,dv\right)\,|d\alpha(u)|=
\int_{z/c}^{\infty}\left(\int_{z}^{cv}|d\alpha(u)|\right)\frac{|\varphi(v)|}{v}\,dv.
\]
To estimate $I_{2}$, using
\[
|\alpha(u)|=\left|\alpha(z)+\int_{z}^{u}\,d\alpha(v)\right|\ll |\alpha(z)|+
\int_{z}^{u}|d\alpha(v)|,\quad u>z,
\]
and condition \eqref{GMcond-usi}, we have
\begin{multline*}
I_{2}\ll
|\alpha(z)|\int_{z}^{\infty}|d\varphi(u)|+\int_{z}^{\infty}\left(\int_{z}^{u}|d\alpha(v)|\right)|d\varphi(u)|
\\
\ll |\alpha(z)|\int_{z/c}^{\infty}\frac{|\varphi(v)|}{v}\,dv+
\int_{z}^{\infty}\left(\int_{v}^{\infty}|d\varphi(u)|\right)|d\alpha(v)|
\\
\ll |\alpha(z)|\int_{z/c}^{\infty}\frac{|\varphi(v)|}{v}\,dv+
\int_{z}^{\infty}\left(\int_{v/c}^{\infty}\frac{|\varphi(u)|}{u}\,du\right)|d\alpha(v)|
\\
=|\alpha(z)|\int_{z/c}^{\infty}\frac{|\varphi(v)|}{v}\,dv+
\int_{z/c}^{\infty}\left(\int_{z}^{cu}|d\alpha(v)|\right)\frac{|\varphi(u)|}{u}\,du.
\end{multline*}
Therefore, since
\[
|\alpha(z)|=\left|\int_{0}^{z}\,d\alpha(v)\right|\le \int_{0}^{z}|d\alpha(v)|,
\]
we arrive at
\[
I\ll I_{1}+I_{2}\ll
\int_{z/c}^{\infty}\left(|\alpha(z)|+\int_{z}^{cu}|d\alpha(v)|\right)\frac{|\varphi(v)|}{v}\,dv\le
\int_{z/c}^{\infty}\left(\int_{0}^{cu}|d\alpha(v)|\right)\frac{|\varphi(v)|}{v}\,dv.
\]
Finally, the integral condition on $\alpha$ concludes the proof of \eqref{Iv1}.
\end{proof}

Let $\wh{GM}{}^{d}$, $d\ge 1$, be the collection of all radial functions
$f(x)=f_{0}(|x|)$, $x\in \R^{d}$, which are defined in terms of the inverse Fourier--Hankel transform
\begin{equation}\label{deff0}
f_{0}(z)=\frac{|S^{d-1}|}{(2\pi)^{d}}\int_{0}^{\infty}F_{0}(s)j_{d/2-1}(zs)s^{d-1}\,ds,
\end{equation}
where the function $F_{0}\in GM$ and satisfies the following condition
\begin{equation}\label{GMcond}
\int_{0}^{1}s^{d-1}|F_{0}(s)|\,ds+
\int_{1}^{\infty}s^{(d-1)/2}\,|dF_{0}(s)|<\infty.
\end{equation}
Applying Lemma 1 from the paper \cite{GLT} to $F_{0}$, we obtain that the integral in \eqref{deff0}
converges in the improper sense
and therefore $f_{0}(z)$ is continuous for $z>0$.
In addition, $F_{0}$ is a radial component of the Fourier transform of the function
$f$, that is, $\wh{f}(\xi)=F_{0}(|\xi|)$, $\xi\in \R^{d}$.

Let us give some examples of functions from the class $\wh{GM}{}^{d}$.

\begin{example}
Let $f\in C(\R^{d})\cap L^{p}(\R^{d})$, where $1\le p<2d/(d+1)$ for $d\ge 2$
and $p=1$ for $d=1$, be a radial positive-definite function such that $F_{0}\in
GM$. Then $f\in \wh{GM}{}^{d}$. Indeed, $\wh{f}$ is continuous function
vanishing at infinity and $\wh{f}\ge 0$ \cite[Ch.~1]{SW}. From continuity of
$f$ at zero we get $\wh{f}\in L^{1}(\R^{d})$ \cite[Cor.~1.26]{SW}, i.e.,
$\int_{0}^{\infty}s^{d-1}|F_{0}(s)|\,ds<\infty$. Since any GM-function $F_{0}$
satisfies (\cite[p. 111]{GLT})
\[
\int_{1}^{\infty}s^{\sigma}\,|dF_{0}(s)|\ll
\int_{1/c}^{\infty}s^{\sigma-1}|F_{0}(s)|\,ds,\quad \sigma\ge 0,
\]
then, using $(d-1)/2-1<d-1$, we get
\[
\int_{0}^{1}s^{d-1}|F_{0}(s)|\,ds+
\int_{1}^{\infty}s^{(d-1)/2}\,|dF_{0}(s)|\ll
\int_{0}^{\infty}s^{d-1}|F_{0}(s)|\,ds<\infty.
\]
Therefore, condition \eqref{GMcond} holds, that is, $f\in \wh{GM}{}^{d}$. As an
example of such function we can take $f(x)=(1+|x|^2)^{-(d+1)/2}$ and the
corresponding $F_{0}(s)= c_{d} e^{-s}$.
\end{example}

\begin{example}\label{exa2-GM}
Take $f(x)=j_{d/2}(|x|)$ (for $d=1$, $f(x)=\frac{\sin x}{x}$). Then
$F_{0}(s)=c\chi_{1}(s)\in GM$ and condition\eqref{GMcond} holds, i.e., $f\in \wh{GM}{}^{d}$. Moreover, we have (see, e.g., \cite{GLT})
\[
j_{d/2}(z)\asymp 1,\quad 0\le z\le 1,\qquad
|j_{d/2}(z)|\ll z^{-(d+1)/2},\quad z\ge 1,
\]
and
\[
|j_{d/2}(z)|\gg z^{-(d+1)/2},\quad z\in \bigcup_{k=1}^{\infty}
\left[\rho_{d/2,k}+\varepsilon,\rho_{d/2,k+1}-\varepsilon\right],
\]
where $\rho_{\alpha,k}$~positive zeros of the Bessel function $J_{\alpha}$,
$\inf_{k\ge 1}\left(\rho_{d/2,k+1}-\rho_{d/2,k}\right)\ge 3\varepsilon>0$.
This implies $f\in L^{p}(\R^{d})$ if $p>\frac{2d}{d+1}$.
\end{example}

\begin{example}\label{exa3-GM}
Let $F_{0}(s)\in GM$ and $|\xi|^{d(1-1/p-1/q)}F_{0}(|\xi|)\in L^{q}(\R^{d})$,
$1<q\le p<\infty$, $\frac{2d}{d+1}<p$. Then, using statement~(A.1) below,
condition \eqref{GMcond} for $F_0$ holds, $f$ is defined by (\ref{deff0}), and
$f\in \wh{GM}{}^{d}\cap L^{p}(\R^{d})$. The fact that $f\in L^{p}(\R^{d})$
follows from Pitt's inequality \eqref{GLT1} (take $\beta=0$).
\end{example}

\subsection{Two-sided inequalities}
\begin{theorem}\label{t3}
Let $f\in \wh{GM}{}^{d}\cap L^{p}(\R^{d})$, $d\ge 1$.

\smallbreak
\noindent
\textup{(A)} If $\wh{f}\ge 0$ and $1<p\le q<\infty$, then
\[
\left(\int_{\R^d}\left[\min(1,t|\xi|)^{\theta l}|\xi|^{d(1-1/p-1/q)}|\wh{f}(\xi)|\right]^{q}\,d\xi
\right)^{1/q}\ll
\Omega_l(f,t)_p.
\]

\smallbreak
\noindent
\textup{(B)} If $|\xi|^{d(1-1/p-1/q)}\wh{f}(\xi)\in L^{q}(\R^{d})$, $1<q\le
p<\infty$, $\frac{2d}{d+1}<p$, then
\[
\left(\int_{\R^d}\left[\min(1,t|\xi|)^{\theta l}|\xi|^{d(1-1/p-1/q)}|\wh{f}(\xi)|\right]^{q}\,d\xi
\right)^{1/q}\gg
\Omega_l(f,t)_p.
\]
\end{theorem}

\begin{remark*}
Conditions on $p$ and $q$ in Theorem \ref{t3} (A, B) are less restrictive than
corresponding conditions in Theorem \ref{t2}. It is clear for (A). Since
$\frac{2d}{d+1}\le d\left(\frac{d+1}{2}-\frac{1}{q}\right)^{-1}$, conditions
$q\le p$ and $\frac{2d}{d+1}<p$ in Theorem \ref{t3} (B) are weaker than
$\max\left\{2,q,d\left(\frac{d+1}{2}-\frac{1}{q}\right)^{-1}\right\}\le p$,
which is the corresponding condition in Theorem \ref{t2} (B).
\end{remark*}

In case of $p=q$ Theorem \ref{t3} gives the following equivalence result.

\begin{corollary}\label{eqiv}
If $f\in \wh{GM}{}^{d}\cap L^{p}(\R^{d})$,
$d\ge 1$, $\wh{f}\ge 0$, $\frac{2d}{d+1}<p<\infty$, then
\[
\left(\int_{\R^d}\left[\min(1,t|\xi|)^{\theta
l}|\xi|^{d(1-2/p)}|\wh{f}(\xi)|\right]^{p}\,d\xi\right)^{1/p}\asymp
\Omega_l(f,t)_p.
\]
\end{corollary}

\begin{example*} Take $f(x)=j_{d/2}(|x|)$ (see Example \ref{exa2-GM}). By Corollary \ref{eqiv}, for $0<t<1$ and
$\frac{2d}{d+1}<p<\infty$, we have
\[
\Omega_{l}(f,t)_{p}\asymp \left\|\min(1,t|\xi|)^{\theta
l}|\xi|^{d(1-2/p)}\chi_{1}(|\xi|)\right\|_{p}\asymp t^{\theta l}.
\]
\end{example*}

\subsection{Weighted Fourier inequalities}
To prove Theorem \ref{t3}, we will use several auxiliary results from the
paper~\cite{GLT}.

Let $d\ge 1$, $1<p,q<\infty$,
$\beta-\gamma=d\left(1-\frac{1}{p}-\frac{1}{q}\right)$, $g(x)=g_0(|x|)$, and
$\wh{g}(\xi)=G_{0}(|\xi|)$.

\smallbreak
\noindent
\textbf{(A.1)} If $g_{0}\in GM$, $p\le q$, and
\[
\frac{d}{q}-\frac{d+1}{2}<\gamma<\frac{d}{q},
\]
then the following Pitt's inequality holds \cite[Th.~2 (A)]{GLT}
\[
\left\||\xi|^{-\gamma}\wh{g}(\xi)\right\|_{q}\ll
\left\||x|^{\beta}g(x)\right\|_{p}.
\]
Then changing variables $g\leftrightarrow \wh{f}$, $p\leftrightarrow q$, and
$\beta\leftrightarrow -\gamma$, we get
\begin{equation}\label{GLT1}
\left\||x|^{\beta}f(x)\right\|_{p}\ll
\left\||\xi|^{-\gamma}\wh{f}(\xi)\right\|_{q},\quad
\frac{d}{p}-\frac{d+1}{2}<-\beta<\frac{d}{p},\quad q\le p.
\end{equation}
Here $\wh{f}(\xi)=F_{0}(|\xi|)$ and $F_{0}\in GM$.
Note \cite[Sect. 5.1]{GLT} that the condition $|\xi|^{-\gamma}\wh{f}(\xi)\in L^{q}(\R^{d})$ implies condition \eqref{GMcond}.

\smallbreak
\noindent \textbf{(A.2)} Let $g_{0}\in GM$, $g_{0}\ge 0$ and $g_{0}$ satisfy condition
\eqref{GMcond}. Then if $q\le p$ and
\[
\frac{d}{q}-\frac{d+1}{2}<\gamma,
\]
then \cite[Th. 2(B)]{GLT}
\[
\left\||\xi|^{-\gamma}\wh{g}(\xi)\right\|_{q}\gg
\left\||x|^{\beta}g(x)\right\|_{p}.
\]
Again, changing variables $g\leftrightarrow \wh{f}$, $p\leftrightarrow q$, and
$\beta\leftrightarrow -\gamma$, we arrive at
\begin{equation}\label{GLT}
\left\||x|^{\beta}f(x)\right\|_{p}\gg
\left\||\xi|^{-\gamma}\wh{f}(\xi)\right\|_{q},\quad
\frac{d}{p}-\frac{d+1}{2}<-\beta,\quad p\le q.
\end{equation}
Here $\wh{f}(\xi)=F_{0}(|\xi|)\ge 0$ and $F_{0}\in GM$.

From (A.1) and (A.2) (see also \cite[Th. 1]{GLT}), for a non-negative GM-function
$F_{0}$ satisfying condition \eqref{GMcond}, we have
\begin{equation}\label{aseqhff}
\left\||\xi|^{d(1-2/p)}\wh{f}(\xi)\right\|_{p}\asymp \left\|f(x)\right\|_{p},\quad
\frac{2d}{d+1}<p<\infty.
\end{equation}

\smallbreak
\noindent
\textbf{(A.3)} Let $g_{0}\ge 0$. For $z>0$ we get (see \cite[formula (53)]{GLT})
\begin{equation}\label{GLT3}
\int_{z/(bc)}^{\infty}u^{-1}\left(\int_{u/b}^{bu}\frac{g_{0}(v)}{v}\,dv\right)du\ll
\int_{0}^{2bc/z}u^{(d-1)/2-1}\left(\int_{0}^{u}v^{(d-1)/2}|G_{0}(v)|\,dv\right)du,
\end{equation}
where $1<b<\rho_{d/2,1}$.

\smallbreak
\noindent
\textbf{(A.4)} The following inequality was shown in \cite[pp. 115-116]{GLT}
\begin{multline*}
\left[\int_{0}^{\infty}u^{-\gamma p+dp/q-dp-1}\left(\int_{0}^{u}v^{(d-1)/2-1}
\left(\int_{0}^{v}z^{(d-1)/2}|G_{0}(z)|\,dz\right)dv\right)^{p}\,du
\right]^{1/p}
\\
\ll \left(\int_{\R^{d}}\left[|x|^{-\gamma}|\wh{g}(x)|\right]^{q}\,dx\right)^{1/q},\quad
\frac{d}{q}-\frac{d+1}{2}<\gamma,\quad
q\le p.
\end{multline*}
Noting $u^{-\gamma p+dp/q-dp-1}=u^{-p\beta-d-1}$ and changing variables $\wh{g}\leftrightarrow f$, $p\leftrightarrow q$, $\beta\leftrightarrow
-\gamma$, we obtain
\begin{multline}\label{GLT4}
\left[\int_{0}^{\infty}u^{q\gamma-d-1}\left(\int_{0}^{u}v^{(d-1)/2-1}
\left(\int_{0}^{v}z^{(d-1)/2}|f_{0}(z)|\,dz\right)dv\right)^{q}\,du
\right]^{1/q}
\\
\ll \left(\int_{\R^{d}}\left[|x|^{\beta}|f(x)|\right]^{p}\,dx\right)^{1/p},\quad
\frac{d}{p}-\frac{d+1}{2}<-\beta,\quad p\le q.
\end{multline}

\subsection{Proof of Theorem \ref{t3} in the case $d\ge 2$.}
Let $t>0$, $f\in \wh{GM}{}^{d}\cap L^{p}(\R^{d})$, $f(x)=f_{0}(|x|)$, and
$\wh{f}(\xi)=F_{0}(|\xi|)$. Note that $F_{0}\in GM$. We use notations from the
proof of Theorem \ref{t1}.

First, we prove (B). Let $|\xi|^{d(1-1/p-1/q)}\wh{f}(\xi)\in L^{q}(\R^{d})$. We
have
\begin{multline*}
I=
\left\|\min(1,t|\xi|)^{2l}|\xi|^{d(1-1/p-1/q)}|\wh{f}(\xi)|\right\|_{q}
\\
\asymp
\left\||\xi|^{d(1-1/p-1/q)}\left[1-\eta_{l,b}(t|\xi|)\right]|\wh{f}(\xi)|\right\|_{q}+
\left\||\xi|^{d(1-1/p-1/q)}(t|\xi|)^{2l}\eta_{l,b}(t|\xi|)|\wh{f}(\xi)|\right\|_{q}=:I_{1}+I_{2}.
\end{multline*}
Then inequalities
\begin{multline*}
\int_{0}^{cu}|d(1-\eta_{l,b}(tv))|\asymp
t^{2l}\int_{0}^{cu}v^{2l-1}\eta_{l,b-1}(tv)\,dv\le
t^{2l}\int_{0}^{\min(cu,1/t)}v^{2l-1}\,dv
\\
\asymp
\min(1,ctu)^{2l}\asymp 1-\eta_{l,b}(tu),\quad b>1,
\end{multline*}
and Lemma \ref{lemGMcond} imply that the function
$\left[1-\eta_{l,b}(ts)\right]F_{0}(s)=\left[1-\eta_{l,b}(t|\xi|)\right]\!\wh{f}(\xi)$
is a GM-function.
Using Pitt's inequality \eqref{GLT1} for
$\beta=0$ and $\gamma=d\left(\frac{1}{p}+\frac{1}{q}-1\right)$ yields
\begin{equation}\label{est1}
I_{1}=\left\||\xi|^{d(1-1/p-1/q)}\left[f-R_{\lambda,l,b}(f)\right]\!\wh{\;\;}(\xi)\right\|_{q}\gg
\left\|f-R_{\lambda,l,b}(f)\right\|_{p}
\end{equation}
for
\begin{equation}\label{p}
p>\frac{2d}{d+1},\quad q\le p.
\end{equation}

Since $\eta_{l,b}(s)=0$ when $s\ge 1$, then
$(ts)^{2l}\eta_{l,b}(ts)=\min(1,ts)^{2l}\eta_{l,b}(ts)$. This and
\eqref{hatR} give
\[
(-1)^{l}t^{2l}\left[\Delta^l
R_{\lambda,l,b}(f)\right]\!\wh{\;\;}(\xi)=\eta_{l,b}(ts)\min(1,ts)^{2l}F_{0}(s),\quad
s=|\xi|.
\]
Also, since $\eta_{l,b}(t|\xi|)=\wh{G_{l,\lambda,b}}(\xi)$, then
\[
(-1)^{l}t^{2l}\Delta^l R_{\lambda,l,b}(f)=G_{\lambda,l,b}*h,\quad
\wh{h}(\xi)=\min(1,t|\xi|)^{2l}F_{0}(|\xi|).
\]
Using Young's convolution inequality, we obtain
\[
\left\|t^{2l}\Delta^l R_{\lambda,l,b}(f)\right\|_{p}\le
\|G_{\lambda,l,b}\|_{1}\|h\|_{p}=
\|G_{l,b}\|_{1}\|h\|_{p}\ll \|h\|_{p}.
\]

We remark that
\begin{equation}\label{minF}
\min(1,ts)^{2l}F_{0}(s)\in GM.
\end{equation}
This follows from the estimate
\[
\int_{0}^{cu}|d\min(1,tv)^{2l}|\asymp
t^{2l}\int_{0}^{\min(cu,1/t)}v^{2l-1}\,dv\asymp
\min[(ctu)^{2l},1]\asymp \min(1,tu)^{2l},
\]
and Lemma \ref{lemGMcond}.

Using again Pitt's inequality \eqref{GLT1}, we have
\begin{equation}\label{est2}
I=\left\||\xi|^{d(1-1/p-1/q)}\wh{h}(\xi)\right\|_{q}\gg
\|h\|_{p}\gg \left\|t^{2l}\Delta^l R_{\lambda,l,b}(f)\right\|_{p}.
\end{equation}

Adding estimates \eqref{est1} and \eqref{est2}, we get
\[
\|f - V_{l,t} f\|_p\asymp \left\|f-R_{\lambda,l,b}(f)\right\|_{p}+
t^{2l}\left\|\Delta^l R_{\lambda,l,b}(f)\right\|_{p}\ll I_{1}+I\ll I.
\]
This and \eqref{p} give the part (B) of the theorem.

Let us now prove the part (A).
If $p\le \frac{2d}{d+1}$, the proof follows from Theorem \ref{t2}.
Suppose $\wh{f}(\xi)=F_{0}(|\xi|)\ge 0$.
By \cite[Lemma~3.4]{DD},
\[
\left[f-V_{l,t}f\right]\!\wh{\;\;}(\xi)=\left[1-m_{l}(t|\xi|)\right]\!\wh{f}(\xi),
\]
where the function $m_{l}(s)$ satisfies for $d\ge 2$ the following conditions
\[
0<C_{1}s^{2l}\le 1-m_{l}(s)\le C_{2}s^{2l},\quad 0<s\le \pi,\qquad
0<m_{l}(s)\le v_{d,l}<1,\quad s\ge \pi.
\]
This gives
\begin{equation}\label{mlmin}
1-m_{l}(s)\asymp \min(1,s)^{2l},\quad s\ge 0.
\end{equation}

Define $h(x)=f(x)-V_{l,t}f(x)$ and its radial component by $h_{0}:=G_{0}$.
Using \eqref{GLT3} for the non-negative function
$g_{0}(s)=\left[1-m_{l}(ts)\right]F_{0}(s)$, we obtain
\begin{equation}\label{J}
J(z):=\int_{z/(bc)}^{\infty}u^{-1}\left(\int_{u/b}^{bu}\frac{g_{0}(v)}{v}\,dv\right)du\ll
\int_{0}^{2bc/z}u^{(d-1)/2-1}\left(\int_{0}^{u}v^{(d-1)/2}|h_{0}(v)|\,dv\right)du.
\end{equation}
Using \eqref{mlmin}, we get
\[
J(z)\asymp
\int_{z/(bc)}^{\infty}u^{-1}\left(\int_{u/b}^{bu}\frac{\min(1,tv)^{2l}F_{0}(v)}{v}\,dv\right)du,\quad
z>0.
\]
where, by \eqref{minF}, $\min(1,tv)^{2l}F_{0}(v)\in GM$. Therefore,
\eqref{GLT3-1} for $z>0$ yields
\[
\min(1,tz)^{2l}F_{0}(z)\ll J(z).
\]
Further, the latter and \eqref{J} imply
\begin{multline*}
I=\left\|\min(1,t|\xi|)^{2l}|\xi|^{d(1-1/p-1/q)}|\wh{f}(\xi)|\right\|_{q}
\asymp
\left(\int_{0}^{\infty}\left[z^{d(1-1/p-1/q)}\min(1,tz)^{2l}F_{0}(z)\right]^{q}z^{d-1}\,dz\right)^{1/q}
\\
\ll
\left(\int_{0}^{\infty}\left[z^{d(1-1/p-1/q)}J(z)\right]^{q}z^{d-1}\,dz\right)^{1/q}
\\
\ll
\left(\int_{0}^{\infty}\left[z^{d(1-1/p-1/q)}
\left(\int_{0}^{2bc/z}u^{(d-1)/2-1}\left(\int_{0}^{u}v^{(d-1)/2}|h_{0}(v)|\,dv\right)du\right)
\right]^{q}z^{d-1}\,dz\right)^{1/q}.
\end{multline*}
Changing variables $2bc/z\to z$, we obtain
\begin{equation}\label{Ill}
I
\ll
\left(\int_{0}^{\infty}z^{-qd(1-1/p-1/q)-d-1}
\left[
\int_{0}^{z}u^{(d-1)/2-1}\left(\int_{0}^{u}v^{(d-1)/2}|h_{0}(v)|\,dv\right)du\right]^{q}\,dz\right)^{1/q}.
\end{equation}
Let us now use \eqref{GLT4} for $\beta=0$ and
$\gamma=d\left(\frac{1}{p}+\frac{1}{q}-1\right)$. Since, in this case
\[
z^{-qd(1-1/p-1/q)-d-1}=z^{q\gamma-d-1}
\]
inequalities \eqref{Ill} and \eqref{GLT4} give
\begin{equation}\label{d2-Illh}
I\ll \left(\int_{\R^{d}}|h(x)|^{p}\,dx\right)^{1/p}=
\left\|f-V_{l,t}f\right\|_{p}.
\end{equation}
when $\frac{d}{p}-\frac{d+1}{2}<0$ and $p\le q$. The latter is $\frac{2d}{d+1}<p\le
q$. The proof of (A) is now complete. \hfill\qed

\subsection{Proof of Theorem \ref{t3} in the case $d=1$.}
We follow the proof of Theorem \ref{t1}. We have
\begin{equation*}
\omega_{l}(f,t)_{p}\asymp
\inf\left(\|f-g\|_{p}+t^{l}\|g^{(l)}\|_{p}\colon
g^{(l)}\in \mathrm{E}_{\lambda}\cap L^{p}(\R)\right),\quad \lambda=1/t.
\end{equation*}
To show the estimate of $\omega_{l}(f,t)_{p}$ from above, that is, to prove (B),
 we take $g_{\lambda}(x)$ such that
\[
\wh{g}_{\lambda}(\xi)=\left[1-(t|\xi|)^{l}\right]_{+}^{b}\wh{f}(\xi),\quad b\ge 3.
\]
Note that the function $g_{\lambda}$ is analogues to the Riesz-type means
$R_{\lambda,l,b}(f)$ and satisfies all required properties
\eqref{defeta}--\eqref{etamin} with $l$ in place of $2l$. In particular,
$1-\left[1-(ts)^{l}\right]_{+}^{b}\asymp \min(1,ts)^{l}$. Proceeding similarly to
the proof of (B) in the case $d\ge 2$, we arrive at the statement (B) in the
case $d=1$.

 Let us now show (A). Let $\frac{2d}{d+1}<p\le q<\infty$ and $\wh{f}\ge 0$. Equivalence \eqref{omega2} gives
\[
\omega_{l}(f,t)_{p}\asymp
\left\|f-S_{\lambda}(f)\right\|_{p}+t^{l}\left\|D_{l}S_{\lambda}(f)\right\|_{p}\ge
\left\|h\right\|_{p},
\]
where $h=f-S_{\lambda}(f)+t^{l}D_{l}S_{\lambda}(f)$.
Moreover,
$\wh{h}(\xi)=\left[1-\chi_{\lambda}(|\xi|)+(t|\xi|)^{l}\chi_{\lambda}(|\xi|)\right]\wh{f}(\xi)$.
Because of \eqref{thm1-min} and \eqref{minF} with $s\ge 0$, we have
$\wh{h}(\xi)=\min(1,ts)^{l}F_{0}(s)\in GM$. Using then \eqref{GLT} with
$\beta=0$, we obtain
\[
\omega_{l}(f,t)_{p}\gg \left\|h\right\|_{p}\gg
\left\||\xi|^{1-1/p-1/q}\wh{h}(\xi)\right\|_{p}=
\left\|\min(1,t|\xi|)^{l}|\xi|^{1-1/p-1/q}\wh{f}(\xi)\right\|_{p}.
\tag*{\qed}
\]

\section{Growth of Fourier coefficients via moduli of smoothness. The case of functions on $\mathbb{T}^d$}
\label{sec-td}

Let $f\in L^{p}(\T^{d})$, $1<p<\infty$, and
\[
\wh{f}_{n}=\int_{\T^{d}}f(x)e^{inx}\,dx,\quad n\in \Z^{d},\qquad
\bigl\|\wh{f}_{n}\bigr\|_{l^{q}(\Z^{d})}:=\biggl(\sum_{n\in
\Z^{d}}\bigl|\wh{f}_{n}\bigr|^{q}\biggl)^{1/q}.
\]
In the paper \cite[Th. 4.1]{Di} the following was proved
\[
\left\|\min(1,t|n|)^{\theta
l}\bigl|\wh{f}_{n}\bigl|\right\|_{l^{p'}(\Z^{d})}
\ll
\Omega_l(f,t)_p,\quad 1<p\le 2,
\]
where $\Omega_l(f,t)_p$ is given by \eqref{Ot-d2} and \eqref{Ot-d1} with
$\|\cdot\|_p=\|\cdot\|_{L^p(\T^{d})}$.

The goal of the section is to obtain the generalization of this result which is
a periodic analogue of inequalities \eqref{bray1}--\eqref{bray2}.

\begin{theorem}\label{tT}
Let
$f\in L^{p}(\T^{d})$, $d\ge 1$, $1<q<\infty$ and
$\gamma=d\left(\frac{1}{p}+\frac{1}{q}-1\right)$.

\textup{(A)} Let $1<p\le 2$. Then for $p\le q\le p'$ we have
$\bigl\{(1+|n|)^{-\gamma}\wh{f}_{n}\bigr\}
\in {l^{q}(\Z^{d})}
$, and
\begin{equation}\label{modT}
\left\|\min(1,t|n|)^{\theta l}(1+|n|)^{-\gamma}\bigl|\wh{f}_{n}\bigl|\right\|_{l^{q}(\Z^{d})}\ll
\Omega_l(f,t)_p.
\end{equation}

\textup{(B)} Let $2\le p<\infty$, $\bigl\{(1+|n|)^{-\gamma}\wh{f}_{n}\bigr\}
\in {l^{q}(\Z^{d})}$, and $\max\left\{q,q'\right\}\le p$. Then
\begin{equation}\label{modT11}
\left\|\min(1,t|n|)^{\theta l}(1+|n|)^{-\gamma}\bigl|\wh{f}_{n}\bigl|\right\|_{l^{q}(\Z^{d})}\gg
\Omega_l(f,t)_p.
\end{equation}
\end{theorem}

The proof of this theorem is similar to the proof of estimates
\eqref{bray1}-\eqref{bray2} from Theorem \ref{t1}. The key points are Pitt's
inequalities of form
\begin{equation}\label{PittT}
\left\|\wh{f}_{n}(1+|n|)^{-\gamma}\right\|_{l^{q}(\Z^{d})}\ll \|f\|_{L^{p}(\T^{d})}, \qquad 1<p\le 2
\end{equation}
and
\begin{equation}\label{PittTTT}
\left\|\wh{f}_{n}(1+|n|)^{-\gamma}\right\|_{l^{q}(\Z^{d})}\gg \|f\|_{L^{p}(\T^{d})}, \qquad p\ge 2,
\end{equation}
under the corresponding conditions on $q$, as well as the realization
results for the $K$-functionals in the periodic case (see \cite{Di} and
\cite{DIT}).

\begin{proof}[Proof of \eqref{PittT}]
Let us show that the proof of \eqref{PittT} follows from Pitt's inequality for functions on
$\R^{d}$.
 Note that $\gamma\ge 0$.
 Let $f_{*}$ be the function on $\R^{d}$ such that $f_{*}=f$ on $(-\pi,\pi]^{d}$ and
$f_{*}=0$ outside $(-\pi,\pi]^{d}$.
Then
\[
\|f_{*}\|_{L^{p}(\R^{d})}=\|f\|_{L^{p}(\T^{d})},\qquad
\wh{f_{*}}(\xi)=\int_{\T^{d}}f(x)e^{i\xi x}\,dx,\quad \xi\in \R^{d},\qquad
\wh{f_{*}}(n)=\wh{f}_{n},\quad n\in \Z^{d}.
\]

Further, we use the results from \cite[Ch.~3]{Ni}. For an entire function $g$ of
exponential type $\sigma e$, $\sigma>0$, we have
\begin{equation}\label{ineqg}
\|g\|_{l^{q}(\Z^{d})}\le (1+\sigma)^{d}\|g\|_{L^{q}(\R^{d})},\quad q\ge 1.
\end{equation}

Note that the function $\wh{f_{*}}$ is an entire function of exponential type
$\pi \ol{e}$, where $\ol{e}=(1,\ldots,1)\in \R^{d}$. We cannot use
\eqref{ineqg} since the weight function $|\xi|^{-\gamma}$, $\gamma\ge 0$, is
not an entire function. However, it is possible to construct a positive radial
entire function of exponential (spherical) type such that for $|\xi|\ge 1$ this
function is equivalent to $|\xi|^{-\gamma}$.

We consider
\[
\psi_{\gamma}(u)=j_{\nu}\Bigl(\frac{u+i}{2}\Bigr)j_{\nu}\Bigl(\frac{u-i}{2}\Bigr),\qquad
u\in \mathbb{C},\quad 2\nu+1=\gamma\ge 0,
\]
where $j_{\nu}$ is the normalized Bessel function. The function $\psi_{\gamma}$
is an even positive entire function of type~1. Positivity of $\psi_{\gamma}$
follows from the fact that all its zeros lie on lines $t\pm i$, $t\in \R$.
The asymptotic expansion of Bessel
functions \cite[formula 9.2.1]{AS} yields, for $|z|\to\infty$,
\[
j_{\nu}(z)=\frac{C_{\nu}}{z^{\nu+1/2}}
\left(\cos\left(z-c_{\nu}\right)+O(|z|^{-1})\right),\qquad
\Re z\ge 0,\quad |\Im z|\ll 1.
\]
This and $\psi_{\gamma}(0)>0$ give $\psi_{\gamma}(u)\asymp (1+|u|)^{-\gamma}$,
$u\in \R$.

Let us now consider the radial function $\psi_{\gamma}(|\xi|)$, $\xi\in
\R^{d}$, which is an entire function of (spherical) type~1, and therefore, of
type $\ol{e}$. Also,
\begin{equation}\label{psi}
\psi_{\gamma}(|\xi|)\asymp (1+|\xi|)^{-\gamma},\quad \xi\in \R^{d}.
\end{equation}

Define $g(\xi)=\wh{f_{*}}(\xi)\psi_{\gamma}(|\xi|)$, which is an entire
function of type $(\pi+1)\ol{e}$. Using \eqref{psi}, we get
\[
\|g\|_{l^{q}(\Z^{d})}=\biggl(\sum_{n\in
\Z^{d}}\left|\wh{f_{*}}(n)\psi_{\gamma}(|n|)\right|^{q}\biggr)^{1/q}\asymp
\biggl(\sum_{n\in
\Z^{d}}\left|\wh{f}_{n}(1+|n|)^{-\gamma}\right|^{q}\biggr)^{1/q},
\]
\[
\|g\|_{L^{q}(\R^{d})}=\biggl(\int_{\R^{d}}\left|\wh{f_{*}}(\xi)\psi_{\gamma}(|\xi|)\right|^{q}\,d\xi\biggr)^{1/q}\ll
\biggl(\int_{\R^{d}}\left|\wh{f_{*}}(\xi)|\xi|^{-\gamma}\right|^{q}\,d\xi\biggr)^{1/q}.
\]
Then by \eqref{ineqg} and Pitt's inequality for function on $\R^{d}$,
we have
\begin{align*}
\left\|\wh{f}_{n}(1+|n|)^{-\gamma}\right\|_{l^{q}(\Z^{d})}
&\asymp
\|g\|_{l^{q}(\Z^{d})}\le (\pi+2)^{d}\|g\|_{L^{q}(\R^{d})}\ll
\left\|\wh{f_{*}}(\xi)|\xi|^{-\gamma}\right\|_{L^{q}(\R^{d})}
\\
&\ll \|f_{*}\|_{L^{p}(\R^{d})}=\|f\|_{L^{p}(\T^{d})}.
\end{align*}
Thus we have proved the Pitt inequality \eqref{PittT} for function on $\T^{d}$.
\end{proof}
\begin{proof}[Proof of \eqref{PittTTT}]
The following inequality is a consequence of \cite[Th. 7]{nurs} and Hardy's
inequality for rearrangements:
\begin{equation}\label{erlan}
\|f\|_{L_p}\ll \biggl(\sum_{k\in \mathbb Z^d}\prod_{j=1}^d
(|k_j|+1)^{q/p'-1}|\wh{f}_{k}|^q\biggr)^{1/q},\qquad
\max\left\{q,q'\right\}\le p.
\end{equation}
The latter immediately gives \eqref{PittTTT}. We would like to thank Erlan
Nursultanov for drawing our attention to his result \eqref{erlan}, which simplifies the
proof.
\end{proof}

\section{An equivalence result for periodic functions}

A complex null-sequence $a=\{a_n\}_{n\in \N}$ is said to be \textit{general
monotone}, written $a\in {GM}$, if (see \cite{compt}) there exists $c>1$ such
that ($\Delta a_k=a_k-a_{k+1}$)
\[
\sum_{k=n}^{\infty} |\Delta a_k|\ll \sum_{k=[n/c]}^{\infty}
\frac{|a_k|}{k},\quad n\in \N.
\]

\begin{theorem}\label{t4}
Let $f\in L^p(\mathbb{T})$, $1< p<\infty$, and
\[
f(x)\sim \sum_{n=1}^\infty (a_n \cos nx + b_n \sin nx),
\]
where nonnegative $\{a_n\}_{n\in \N}$, $\{b_n\}_{n\in \N}$ are general monotone sequences. Then
\begin{equation}\label{pl}
\omega_l\left(f, t\right)_p\asymp
\biggl(\sum_{\nu=1}^{\infty}
\min(1,\nu t)^{lp} \nu^{p-2} \left(a_\nu^p+b_\nu^p\right)\biggr)^{1/p}.
\end{equation}
\end{theorem}

We will use the following lemma (see \cite{AW}).

\begin{lemma}\label{lem}
Let $1 < p < \infty $ and let $\sum_{\nu = 1}^{\infty} a_{\nu} \cos \nu x$ be the Fourier series of $f\in L^1(\T)$.

\smallbreak
\noindent
\textup{(A)} If the sequences $\{a_n\}$ and $\{\beta_n\}$ are such that
\begin{equation}\label{cond-coeff}
\sum_{k=\nu}^{\infty} |\Delta a_{k}|\ll
\beta_\nu,\quad \nu\in \mathbb{N},
\end{equation}
then
\begin{equation}\label{cond-norm}
 \|f\|_p^p\ll \sum_{\nu = 1}^{\infty} \nu^{p-2} \beta_{\nu}^p.
\end{equation}

\smallbreak
\noindent
\textup{(B)} If $a=\{a_n\}$ is a nonnegative sequence, then
\begin{equation}\label{ask}
 \sum_{n=1}^\infty \biggl( \sum_{k=[n/2]}^n
a_{k}\biggr)^p n^{-2}\ll \|f\|_{p} ^p.
\end{equation}
\end{lemma}

\begin{proof}[Proof of Theorem \ref{t4}] First, we remark that since
$1< p<\infty$ it is sufficient to prove
that
\[
\omega_l^p\Bigl(f,\frac{1}{n}\Bigr)_p\asymp I_1+I_2,
\]
where
\[
I_1=n^{-lp} \sum_{\nu=1}^{n} a_\nu^p
\nu^{(l+1)p-2},\quad I_2=\sum_{\nu=n+1}^{\infty}
a_\nu^p \nu^{p-2},
\]
\[
f(x)\sim \sum_{n=1}^\infty a_n \cos nx,\quad \{a_n\}_{n\in \N}\in
GM.
\]
We will also use the realization result for the modulus of smoothness (see
\cite{DIT}), that is,
\begin{equation}\label{real}
\omega_l^p\Bigl(f,\frac{1}{n}\Bigr)_p\asymp \bigl\|f(x)-T_n(x)\bigr\|_p^p+
n^{-lp}\bigl\|T_n^{(l)}(x)\bigr\|_p^p,
\end{equation}
where $T_n(f)$ is the $n$-th almost best approximant, i.e.,
$\left\|f(x)-T_n(x)\right\|_p\ll E_n(f)_p.$ In particular we can take $T_n$ as
$S_n=S_n(f)$, i.e., the $n$-th partial sum of $\sum_{k=1}^\infty a_k \cos kx$.

Let us prove estimate of $I_1$ and $I_2$ from above. Since $\{a_n\}\in GM$, we have
\begin{equation}\label{gm-ocenka}
a_\nu\le \sum_{l=\nu}^{\infty} |\Delta a_l|\ll \sum_{l=[\nu/c]}^{\infty} \frac{a_l}{l},
\end{equation}
then H\"{o}lder's inequality yields
\begin{align*}
I_1
&\ll
n^{-lp}\sum_{\nu=1}^{n}\biggl(\sum_{j=[\nu/c]}^{\infty} \frac{a_j}{j}\biggr)^p
\nu^{(l+1)p-2}
\\
&\ll
n^{-lp}
\sum_{\nu=1}^{n}\biggl(\sum_{j=[\nu/c]}^{n} \frac{a_j}{j}\biggr)^p
\nu^{(l+1)p-2}
+
n^{p-1}\biggl(\sum_{j=n}^{\infty} \frac{a_j}{j}\biggr)^p
\\
&\ll
n^{-lp}
\sum_{\nu=1}^{n}\biggl(\sum_{j=\nu}^{n} \frac{a_j}{j}\biggr)^p
\nu^{(l+1)p-2}
+
\sum_{j=n+1}^{\infty} a_j^p j^{p-2}=:I_{3}+I_{2}.
\end{align*}
To estimate $I_2$ and $I_3$, we are going to use the following inequalities
\begin{equation}\label{pachki}
\sum_{s=n}^{\infty}a_s\ll \sum_{s=n}^{\infty}\frac{1}{s} \sum_{m=[s/2]}^{s}a_m\quad\text{and}\quad
\sum_{s=1}^{n}a_s\ll \sum_{s=1}^{2n}\frac{1}{s} \sum_{m=[s/2]}^{s}a_m.
\end{equation}
Then by Hardy's inequality \cite{HLP}, we have
\begin{align*}
I_{3}&\ll
n^{-lp}
\sum_{\nu=1}^{n}\biggl(\sum_{j=\nu}^{2n}\frac{1}{j^2} \sum_{m=[j/2]}^{l}a_m \biggr)^p
\nu^{(l+1)p-2}
\\
&\ll
n^{-lp}
\sum_{j=1}^{2n}\biggl(\sum_{m=[j/2]}^{j}a_m \biggr)^p
j^{lp-2}.
\end{align*}
Then Lemma \ref{lem} (B) and \eqref{real} yield
\[
I_{3}\ll
n^{-lp}
\biggl\|\sum_{\nu=1}^{2n} \nu^l a_\nu \cos \nu x\biggr\|_p^p
\asymp n^{-lp}\bigl\|S^{(l)}_{2n}(f)\bigr\|_p^p\ll \omega_l^p\Bigl(f,\frac{1}{2n}\Bigr)_p
\ll \omega_l^p\Bigl(f,\frac{1}{n}\Bigr)_p.
\]
Further, using \eqref{gm-ocenka}, \eqref{pachki}, and Hardy' inequality, we have
\begin{align*}
I_{2}&\ll
\sum_{j=n+1}^{\infty} j^{p-2} \biggl( \sum_{s=[j/c]}^{\infty} \frac{a_s}{s} \biggr)^p
\ll
\sum_{j=n+1}^{\infty} j^{p-2} \biggl( \sum_{s=[j/c]}^{\infty} \frac{1}{s^2} \sum_{m=[s/2]}^{s} a_m \biggr)^p
\\
&\ll
\sum_{s=[n/c]}^{\infty} s^{-2} \biggl( \sum_{m=[s/2]}^{s} a_m \biggr)^p
\ll
\sum_{s=2n}^{\infty} s^{-2} \biggl( \sum_{m=[s/2]}^{s} a_m \biggr)^p
+
n^{-lp}
\sum_{s=1}^{2n}s^{lp-2}\biggl(\sum_{m=[s/2]}^{s}a_m \biggr)^p.
\end{align*}
The last sum was estimated above. Again, by Lemma \ref{lem} (B) and
\eqref{real},
\[
\sum_{s=2n}^{\infty} s^{-2} \biggl( \sum_{m=[s/2]}^{s} a_m \biggr)^p\ll
\biggl\|\sum_{\nu=n}^{\infty} a_\nu \cos \nu x\biggr\|_p^p
\ll \omega_l^p\left(f,\frac{1}{n}\right)_p.
\]
So, we showed that
\[
I_1+I_2\ll \omega_l^p\Bigl(f,\frac{1}{n}\Bigr)_p.
\]

To prove the reverse, we use Lemma \ref{lem} (A), the definition of the GM
class, H\"{o}lder's and Hardy's inequalities:
\begin{align*}
\|f-S_n\|_p^p&\ll
\sum_{j=1}^{\infty} \beta_j^{p-2} j^{p-2}
\ll
n^{p-1}\biggl(\sum_{s=n}^\infty|\Delta a_s| \biggr)^p+
\sum_{j=n}^{\infty} j^{p-2}\biggl(\sum_{s=l}^\infty|\Delta a_s| \biggr)^p
\\
&\ll
n^{p-1}\biggl(\sum_{s=[n/c]}^\infty\frac{a_s}{s} \biggr)^p+
\sum_{j=n}^{\infty} j^{p-2}\biggl(\sum_{s=[j/c]}^\infty\frac{a_s}{s} \biggr)^p
\ll
\sum_{j=[n/c]}^{\infty} a_j^p j^{p-2}
\ll I_1+I_2,
\end{align*}
where $\beta_j=\sum_{s=\max(j,n)}^\infty|\Delta a_s|$.
Similarly,
\[
n^{-lp}\bigl\|S^{(l)}_{n}(f)\bigr\|_p^p
\ll
n^{-lp}
\biggl\|\sum_{\nu=1}^{n} \nu^l a_\nu \cos \nu x\biggr\|_p^p
\ll
n^{-lp}
\sum_{\nu=1}^{n} \nu^{p-2}\biggl(\sum_{s=\nu}^n|\Delta (s^l a_s)| \biggr)^p.
\]
Further,
\[
\sum_{s=\nu}^n |\Delta (s^l a_s)|
\ll
\sum_{s=\nu}^n s^{l-1} a_s + \sum_{s=\nu}^n s^l |\Delta a_s|
\ll
\sum_{s=\nu}^n s^{l-1} a_s +
\sum_{s=\nu}^n |\Delta a_s|\biggl( \sum_{m=\nu}^{s}{m^{l-1}}+\nu^l \biggr),
\]
and after routine calculations, we arrive at
\[
\sum_{s=\nu}^n |\Delta (s^l a_s)|
\ll
\sum_{s=[\nu/c]}^n s^{l-1}a_s + n^l\sum_{m=n}^\infty\frac{a_m}{m}.
\]
Using this and Hardy's inequality, we get
$n^{-lp}\bigl\|S^{(l)}_{n}(f)\bigr\|_p^p\ll I_1+I_2$. Finally, by \eqref{real},
\[
\omega_l^p\Bigl(f,\frac{1}{n}\Bigr)_p\ll I_1+I_2.
\tag*{\qed}
\]
\def\qed{\relax}\end{proof}

\section{Discussion and applications}

\subsection{Riemann--Lebesgue-type results}

From Theorem A and \cite[Th.~2.2]{Di}, one has the following estimate of the
Fourier transform
\begin{equation}\label{oo}
t^{\theta l}\left(\int_{|\xi|<1/t} |\xi|^{\theta lp'} |\wh{f}(\xi)|^{p'} \,d\xi\right)^{1/p'}
+
\left(\int_{1/t\le |\xi|} |\wh{f}(\xi)|^{p'} \,d\xi\right)^{1/p'}
\ll
 \Omega_l(f, t)_p, \qquad 1< p \le 2.
\end{equation}
On the other hand, Theorem \ref{t1} gives ($p\le q\le p'$, $1< p \le 2$)
\begin{equation}\label{ooo}
t^{\theta l}\left(\int_{|\xi|<1/t} |\xi|^{\theta lq+ dq(1-1/p-1/q)}
|\wh{f}(\xi)|^{q} \,d\xi\right)^{1/q} + \left( \int_{1/t\le |\xi|}
|\xi|^{dq(1-1/p-1/q)} |\wh{f}(\xi)|^{q}\,d\xi \right)^{1/q}\ll \Omega_l(f,t)_p.
\end{equation}
If $q=p'$ \eqref{ooo} reduces to \eqref{oo}. The following example shows that
\eqref{ooo}, in general, provides better estimates than \eqref{oo}.

\begin{example*}
Let $\wh{f}(\xi)=F_{0}(|\xi|)$,
\[
F_{0}(s)=\frac{s^{-d/p'}}{\ln^{2/p}(2+s)},\qquad \frac{2d}{d+1}<p<\infty.
\]
Note that $F_{0}$ is decreasing to zero and therefore $F_{0}\in GM$. Also, it is easy to see that
$|\xi|^{d(1-2/p)} \wh{f}(\xi)\in L^{p}(\R^{d})$.
Hence, as in Example \ref{exa3-GM} (for $q=p$) we get $f\in \wh{GM}{}^{d}\cap L^{p}(\R^{d})$.

We have
\begin{multline*}
t^{\theta l}\left(\int_{|\xi|<1/t} |\xi|^{\theta lq+ dq(1-1/p-1/q)}
|\wh{f}(\xi)|^{q} \,d\xi\right)^{1/q} + \left(\int_{|\xi|\ge
1/t}|\xi|^{dq(1-1/p-1/q)}|\wh{f}(\xi)|^{q}\,d\xi \right)^{1/q}
\\
\asymp
\bigl[\ln(2+1/t)\bigr]^{-2/p+1/q}.
\end{multline*}
Then \eqref{oo} gives
\[
\bigl[\ln(2+1/t)\bigr]^{1-3/p}\ll \Omega_l(f, t)_p,\qquad p\le 2,
\]
and \eqref{ooo} implies (with $q=p$)
\[
\bigl[\ln(2+1/t)\bigr]^{-1/p}\ll \Omega_l(f, t)_p,\qquad p\le 2.
\]
The latter estimate is stronger. Moreover, it is sharp since by Corollary \ref{eqiv} we in fact have
\begin{equation*}
\bigl[\ln(2+1/t)\bigr]^{-1/p}\asymp \Omega_l(f, t)_p,\qquad \frac{2d}{d+1}<p<\infty.
\tag*{\qed}
\end{equation*}
\end{example*}

\subsection{Pointwise Riemann--Lebesgue-type results}
\label{RLTR}

For $f\in L^1(\R^{d})\cap L^p(\R^{d})$, $1<p\le 2$, the Riemann--Lebesgue
inequality
\begin{equation}\label{fc-1}
|\wh{f}(\xi)|\ll \Omega_l(f, 1/|\xi|)_p
\end{equation}
does not hold in general.

Let us consider the case of $d=1$ and $l\ge 2$. Define
\[
f(x)=\sum_{n\in \Z}a_{n}\psi_{n}(x),\qquad
\psi_{n}(x)=\varepsilon_{n}\varphi(\varepsilon_{n}x)e^{-inx},
\]
\[
\varphi(x)=(2\pi)^{-1}\left(\frac{\sin(x/2)}{x/2}\right)^{2},\qquad
\wh{\varphi}(\xi)=(1-|\xi|)_{+},
\]
\[
a_{n}=(1+|n|)^{-3/2},\qquad \varepsilon_{n}=(1+|n|)^{-\alpha p'},\quad
1<\alpha<3/2.
\]
Changing variables, we have
\[
\|\varphi(\varepsilon_{n}x)\|_{q}=\varepsilon_{n}^{-1/q}\|\varphi\|_{q}\asymp
\varepsilon_{n}^{-1/q}.
\]
Hence
\begin{equation}\label{fqq}
\|f\|_{q}\le \sum_{n\in \Z}a_{n}\varepsilon_{n}\|\varphi(\varepsilon_{n}x)\|_{q}\asymp
\sum_{n\in \Z}a_{n}\varepsilon_{n}^{1/q'}\le
\sum_{n\in \Z}(1+|n|)^{-3/2}<\infty,\qquad q\ge 1.
\end{equation}
This implies $f\in L^{1}(\R)\cap L^{p}(\R)$. The Fourier transform of $f$ is written as
\[
\wh{f}(\xi)=\sum_{n\in \Z}a_{n}\wh{\psi}_{n}(\xi),\qquad
\wh{\psi}_{n}(\xi)=\wh{\varphi}\left(\frac{\xi-n}{\varepsilon_{n}}\right).
\]

Let us estimate $\omega_{l}(f,1/t)_{p}$ from above. We will use the realization
result (see \eqref{omega2}) given by
\begin{equation}\label{fqqqq}
\omega_{l}(f,1/t)_{p}\asymp
\left\|f-S_{t}(f)\right\|_{p}+t^{-l}\bigl\|S_{t}^{(l)}(f)\bigr\|_{p},\qquad
t>0,\quad 1<p<\infty.
\end{equation}
Since $\supp\wh{\psi}_{n}\subset [n-\varepsilon_{n},n+\varepsilon_{n}]$,
then
\[
S_{t}(f)(x)=\sum_{|n|\le [t]}a_{n}\psi_{n}(x),\qquad
f(x)-S_{t}(f)(x)=\sum_{|n|>[t]}a_{n}\psi_{n}(x).
\]
The function $\varphi$ and its derivatives are given by
\[
\varphi^{(l)}(x)=\frac{1}{2\pi}\int_{-1}^{1}(1-|\xi|)(-i\xi)^{l}e^{-i\xi
x}\,d\xi,\quad l\in \Z_{+}.
\]
Then $|\varphi^{(l)}(x)|\le 1$, $x\in \R$. For $|x|\ge 1$ we get
\[
\bigl|\varphi^{(l)}(x)\bigr|=
\biggl|(2\pi)^{-1}\sum_{j=0}^{l}\binom{l}{j}\left[\sin^{2}(x/2)\right]^{(l-j)}
\left[(x/2)^{-2}\right]^{(j)}\biggr|\ll
\frac{1}{x^{2}}.
\]
Thus, $\left|\varphi^{(l)}(x)\right|\ll \left(1+x^2\right)^{-1}$ and then
\begin{align*}
\bigl|\psi_{n}^{(l)}(x)\bigr|&=\varepsilon_{n}
\biggl|\sum_{j=0}^{l}\binom{l}{j}\left[e^{-inx}\right]^{(l-j)}
\left[\varphi(\varepsilon_{n}x)\right]^{(j)}\biggr|=
\varepsilon_{n}
\biggl|\sum_{j=0}^{l}\binom{l}{j}(-in)^{l-j}
\varepsilon_{n}^{j}\varphi^{(j)}(\varepsilon_{n}x)\biggr|
\\
&\ll
\frac{\varepsilon_{n}}{1+(\varepsilon_{n}x)^{2}}
\sum_{j=0}^{l}\binom{l}{j}|n|^{l-j}\varepsilon_{n}^{j}=
\frac{\varepsilon_{n}(|n|+\varepsilon_{n})^{l}}{1+(\varepsilon_{n}x)^{2}}.
\end{align*}
Then we arrive at
\[
\|\psi_{n}^{(l)}\|_{q}\ll (|n|+\varepsilon_{n})^{l}\varepsilon_{n}^{1/q'}\ll
(1+|n|)^{l}\varepsilon_{n}^{1/q'}.
\]
Using these relations and proceeding similarly to \eqref{fqq}, we get
\[
\left\|f-S_{t}(f)\right\|_{p}\le \sum_{|n|>[t]}a_{n}\varepsilon_{n}^{1/p'}=
\sum_{|n|>[t]}(1+|n|)^{-3/2-\alpha}\ll t^{-1/2-\alpha},
\]
\[
t^{-l}\bigl\|S_{t}^{(l)}(f)\bigr\|_{p}\ll
t^{-l}\sum_{|n|\le [t]}a_{n}(1+|n|)^{l}\varepsilon_{n}^{1/p'}\ll
t^{-l}\sum_{|n|\le [t]}(1+|n|)^{l-3/2-\alpha}.
\]
For $l\ge 2$, $\alpha<3/2$ we get $l-3/2-\alpha>-1$. Then
$t^{-l}\bigl\|S_{t}^{(l)}(f)\bigr\|_{p}\ll t^{-1/2-\alpha}$. Finally,
\eqref{fqqqq} implies
\[
\omega_{l}(f,1/t)_{p}\ll t^{-1/2-\alpha}.
\]
For any large enough $t\in \N$
\[
\wh{f}(t)=a_{t}\wh{\varphi}(0)=a_{t}\asymp t^{-3/2}.
\]
Since $\varepsilon=\alpha-1 >0$, we finally get
\[
\wh{f}(t)\gg t^\varepsilon \omega_{l}(f,1/t)_{p}. \tag*{\qed}
\]

However, let us remark that for functions from the class $\wh{GM}{}^{d}$ class, it is possible to obtain the pointwise bound of the Fourier transform.
\begin{corollary}\label{typaleb}
Let $f\in \wh{GM}{}^{d}\cap L^{p}(\R^{d})$, $d\ge 1$, $\wh{f}(\xi)=F_0(|\xi|)\ge 0$, and $\frac{2d}{d+1}<p<\infty$. Then
\begin{equation}\label{bes6}
F_0(t) \ll t^{-d/p'} \Omega_l(f,1/t)_p.
\end{equation}
\end{corollary}
\begin{proof}
Since $f\in \wh{GM}{}^{d}$, using \eqref{estv} and H\"{o}lder's inequality, we get
\begin{align}\label{bes4}
F_0(s) &\ll
\int_{s/c}^\infty \frac{F_0(u)}{u} \,du =
\int_{s/c}^\infty {F_0(u)} {u}^{d-d/p-1/p} {u}^{-d+(d+1-p)/p} \,du
\notag
\\
&\ll
s^{d/p-d}\left(\int_{s/c}^\infty {F_0^p(u)} {u}^{dp-d-1} \,du \right)^{1/p}.
\end{align}
Then using Corollary \ref{eqiv}, we have
\begin{equation}\label{bes2}
\Omega_l^p(f,t)_p
\asymp
t^{\theta l p} \int_{0}^{1/t} s^{\theta l p +dp-d-1} F_{0}^p(s) \,ds
+
\int^{\infty}_{1/t} s^{dp-d-1} F_{0}^p(s) \,ds
\end{equation}
and by \eqref{bes4}, we finally get
\[
F_0(t)\ll
t^{d/p-d} \left(\int_{t/c}^\infty {F_0^p(u)} {u}^{dp-d-1} \,du\right)^{1/p} \ll
t^{-d/p'} \Omega_l(f, 1/t)_p.
\tag*{\qed}
\]
\def\qed{\relax}\end{proof}

\subsection{Moduli of smoothness and best approximations: sharp relations}

The following direct and inverse theorems of trigonometric approximation are
well known (see e.g. \cite[p.~210]{DeVL}, \cite[Intr.]{DDT}):
\begin{equation}\label{a3}
\frac{1}{n^l}\left( \sum_{\nu=0}^n (\nu+1)^{\tau l-1}
E^\tau_\nu(f)_p \right)^{1/\tau} \ll \omega_{l}
\Bigl(f,\frac{1}{n}\Bigr)_p \ll \frac{1}{n^l}\left(
\sum_{\nu=0}^n (\nu+1)^{ql-1} E^q_\nu(f)_p
\right)^{1/q},
\end{equation}
where $f\in L^p(\mathbb{T})$, $1<p<\infty$, $l, n\in \mathbf{N}$,
$q=\min(2,p),$ $\tau=\max(2,p)$, $E_n(f)_p$ denotes the $n$-th best
trigonometric approximation of $f$ in $L^p$, and $\omega_l(f,\delta)_p$ is the
$L^p$-modulus of smoothness, see \eqref{mod} with $X=\T$.

We remark that \eqref{a3} is the sharp version of classical Jackson
and weak-type inequalities (\cite[p. 205, 208]{DeVL}) and it can be
written equivalently as follows (\cite{DDT}):
\begin{equation}\label{a-5}
t^l\left( \int_t^1 u^{-\tau l-1} \omega_{l+1}^\tau (f,
u)_p \,du \right)^{1/\tau} \ll \omega_{l} (f,
t)_p \ll {t^l}\left( \int_t^1 u^{-ql-1}
\omega_{l+1}^q (f, u)_p \,du \right)^{1/q}.
\end{equation}
Constructing individual functions shows (\cite{DDT}) that the parameters
$q=\min(2,p)$ and $\tau=\max(2,p)$ are optimal in \eqref{a3} and \eqref{a-5}.
For functions on $[-1,1]$ inequalities of type (\ref{a3}) and (\ref{a-5}) were obtained in \cite{totik, DDT}.

For functions on ${L^{p}(\R^{d})}$, similar results were also proved for
$\Omega_{k}(f,t)_p$ and $E_n(f)_p$, i.e., the best $L^p$-approximation by
functions of exponential type $n$ (see \cite{DDT}). For example, an analogue of
\eqref{a3} is given by
\[
\frac{1}{2^{\theta ln}}
\left(\sum_{\nu=0}^n
{2^{\theta l\tau \nu}} E^\tau_{2^\nu}(f)_p \right)^{1/\tau}
\ll \Omega_{l} \Bigl(f, \frac{1}{2^n} \Bigr)_p
\ll
\frac{1}{2^{\theta l n}}
\left(\sum_{\nu=0}^n
{2^{\theta l q \nu}} E^q_{2^\nu}(f)_p \right)^{1/q},\qquad
\|\cdot\|_p=\|\cdot\|_{L^{p}(\R^{d})}.
\]
Below we show that for functions from the class $\wh{GM}{}^{d}$ we can completely solve
the problem of description of relationships between $\Omega_l(f,t)_p$ and
$E_n(f)_p$ as well as $\Omega_l(f,t)_p$ and $\Omega_{l+1}(f,t)_p$.

\begin{theorem}\label{Marchaud}
If $f\in \wh{GM}{}^{d}\cap L^{p}(\R^{d})$, $d\ge 1$, $\wh{f}\ge 0$, and
$\frac{2d}{d+1}<p<\infty$, then
\begin{equation}\label{eqeq}
\Omega_l(f,t)_p \asymp \left( t^{\theta l p} \int_{t}^1 u^{-\theta l p}
\Omega_{l+1}^p(f,u)_p \frac{du}{u} \right)^{1/p} + t^{\theta l}A(f)_p,\qquad
0<t<\frac12,
\end{equation}
where $A(f)_{p}:=\bigl\||\xi|^{\theta l+d(1-2/p)}\chi_{1}(|\xi|)\wh{f}(\xi)\bigr\|_p\ll \Omega_l(f,1)_p$.
 In particular, we have
\[
\left(
t^{\theta l p}
\int_{t}^1 u^{-\theta l p} \Omega_{l+1}^p(f,u)_p \frac{du}{u} \right)^{1/p}
\ll
\Omega_l(f,t)_p
\ll
\left(
t^{\theta l p}
\int_{t}^1 u^{-\theta l p} \Omega_{l+1}^p(f,u)_p \frac{du}{u} \right)^{1/p} +
t^{\theta l} \|f\|_p
\]
and
\begin{equation}\label{a-vsp}
\frac{1}{2^{\theta l n}}
\left(\sum_{\nu=0}^n
{2^{\theta l p \nu}} E^p_{2^\nu}(f)_p \right)^{1/p}
\ll
\Omega_{l} \Bigl(f, \frac{1}{2^n} \Bigr)_p
\ll
\frac{1}{2^{\theta ln}}
\left(\sum_{\nu=0}^n
{2^{\theta l p \nu}} E^p_{2^\nu}(f)_p \right)^{1/p}
+ \frac{1}{2^{\theta ln}}\| f\|_p.
\end{equation}
\end{theorem}

\begin{remark}\label{remarkremark}
In \eqref{eqeq} one cannot drop $t^{\theta l}A(f)_p$. Indeed, consider
\[
F_{0}^p(s)= s^{-( dp-d-1)} \chi_{1/n}(s).
\]
Then
\[
\Omega_l^p(f,t)_p\asymp t^{\theta l p}\int_0^{1/n} s^{\theta l p + dp-d-1}
F_{0}^p(s) \,ds \asymp t^{\theta l p} \int_{0}^{1/n} s^{\theta l p} \,ds \asymp
t^{\theta l p} n^{-\theta l p-1}.
\]
Using this,
\[
t^{\theta l p}
\int_{t}^1 u^{-\theta l p} \Omega_{l+1}^p(f,u)_p \frac{du}{u}
\asymp
t^{\theta l p}
\int_{t}^1
u^{\theta p} n^{-\theta (l+1) p-1}\frac{du}{u}
\asymp
t^{\theta l p} n^{-\theta (l+1) p-1}.
\]
Hence, writing
\[
t^{\theta l } n^{-\theta l-1/p}
\ll
\Omega_l(f,t)_p
\ll
t^{\theta l }
\left(
\int_{t}^1 u^{-\theta l p} \Omega_{l+1}^p(f,u)_p \frac{du}{u} \right)^{1/p}
\ll
n^{-\theta }
t^{\theta l} n^{-\theta l-1/p}
\]
we arrive at a contradiction as $n\to \infty$.
\end{remark}

\begin{proof}[Proof of Theorem \ref{Marchaud}]
Using Corollary \ref{eqiv}, we get
\[
\Omega_l^p(f,t)_p
\asymp
t^{\theta l p}\int_{0}^{1/t} s^{\theta l p +dp-d-1} F_{0}^p(s) \,ds
+
\int^{\infty}_{1/t} s^{dp-d-1} F_{0}^p(s) \,ds =:J_1(t)+J_2(t)
\]
and
\begin{multline*}
t^{\theta l p}
\int_{t}^1 u^{-\theta l p} \Omega_{l+1}^p(f,u)_p \frac{du}{u}
\asymp
t^{\theta l p}
\int_{1}^{1/t} u^{-\theta p-1}\left[\int_0^u s^{\theta (l+1) p + dp-d-1} F_{0}^p(s) \,ds\right] du
\\
{}+
t^{\theta l p}
\int_{1}^{1/t} u^{\theta l p-1} \left[\int_u^\infty s^{dp-d-1} F_{0}^p(s) \,ds\right] du=:
I_1(t)+I_2(t).
\end{multline*}
Then

\begin{align*}
I_1(t)
&=
t^{\theta l p}
\int_{1}^{1/t} u^{-\theta p-1}\left[ \left(\int_0^1+\int_1^u\right) s^{\theta (l+1) p + dp-d-1} F_{0}^p(s) \,ds\right]\, {du}
\\
&\asymp
t^{\theta l p}
\int_0^1 s^{\theta (l+1) p + dp-d-1} F_{0}^p(s) \,ds
+
t^{\theta l p} \int_{1}^{1/t} s^{\theta (l+1) p + dp-d-1} F_{0}^p(s)
\int_{s}^{1/t} u^{-\theta p-1} \,du\, {ds}
\\
&\ll J_1(t)
\end{align*}
and
\begin{align*}
I_2(t)
&=
t^{\theta l p}
\int_{1}^{1/t} u^{\theta l p-1}\left[ \left(\int_u^{1/t}+\int_{1/t}^\infty \right) s^{ dp-d-1} F_{0}^p(s) \,ds\right]\, {du}
\\
&\asymp
t^{\theta l p} \int_{1}^{1/t} s^{dp-d-1} F_{0}^p(s)
\int^{s}_{1} u^{ \theta l p-1} \,du \,ds
+
 \int_{1/t}^\infty s^{dp-d-1} F_{0}^p(s) \,ds
\\
&\ll J_1(t)+J_2(t).
\end{align*}
Using again Corollary \ref{eqiv},
\begin{align*}
A(f)_{p}&\asymp \left(\int_0^1 s^{\theta lp + dp-d-1} F_{0}^p(s) \,ds\right)^{1/p}\ll
\left(\int_0^\infty s^{dp-d-1}\min(1,s)^{\theta lp} F_{0}^p(s)
\,ds\right)^{1/p}
\\
&\asymp \left\|\min(1,|\xi|)^{\theta l}|\xi|^{d(1-2/p)}\wh{f}(\xi)\right\|_p\asymp
\Omega_l(f,1)_p.
\end{align*}
Moreover, $A^{p}(f)_{p}\ll J_1(t).$ Thus,
\[
I_1(t)+I_2(t)+t^{\theta l p}A^{p}(f)_{p}\ll J_1(t)+J_2(t).
\]

To prove the inverse inequality, we first remark that $s^{-\theta
p}\ll \int_s^{1/t} u^{-\theta p -1} du,$ $1<s<1/(2t)$ and therefore using \eqref{aseqhff},
\begin{align*}
J_1(2t)&\ll
t^{\theta l p}\int_{0}^{1} s^{\theta l p +dp-d-1} F_{0}^p(s) \,ds
+
t^{\theta l p}\int_{1}^{1/2t} s^{\theta (l+1)p+ dp-d-1} F_{0}^p(s)
\left(\int_s^{1/t} u^{-\theta p-1} \,du\right)ds
\\
&\ll
t^{\theta lp}A^{p}(f)_{p}+I_1(t).
\end{align*}
Also,
\begin{align*}
J_2(2t)
&\ll
\int_{1/(2t)}^\infty s^{dp-d-1} F_{0}^p(s)\,ds
\\
&\ll
\int_{1/(2t)}^{1/t} s^{dp-d-1} F_{0}^p(s)\,ds
+
t^{\theta l p}\int_{1/(2t)}^{1/t}
u^{\theta l p-1}
\int_{u}^\infty s^{dp-d-1} F_{0}^p(s)\,ds\,du
\\
&\ll
I_1(t)+I_2(t).
\end{align*}
Finally, to verify \eqref{a-vsp}, we apply \cite[(5.7) and (5.8)]{DDT}.
\end{proof}

Using \eqref{pl}, we state the analogous result for periodic functions; compare with \eqref{a3} and \eqref{a-5}.
\begin{theorem}
Let $f\in L^p(\mathbb{T})$, $1< p<\infty$, and
\[
f(x)\sim \sum_{n=1}^\infty (a_n \cos nx + b_n \sin nx),
\]
where non-negative $\{a_n\}_{n\in \N}$, $\{b_n\}_{n\in \N}$ are general monotone sequences. Then
\[
\omega_l(f,t)_p
\asymp
\left(
t^{l p}
\int_{t}^1 u^{-l p} \omega_{l+1}^p(f,u)_p \frac{du}{u} \right)^{1/p},\qquad 0<t<\frac12.
\]
In particular,
\[
\omega_l(f,1/n)_p
\asymp
\left(
n^{-l p}
\sum_{\nu=0}^{n} (\nu+1)^{l p-1} E_\nu^p(f)_p \right)^{1/p},
\]
where $E_\nu(f)_p$ is the best $L^p$-approximation of $f$ by trigonometric
polynomials of degree $\nu$.
\end{theorem}
Note that similar equivalence results for continuous functions were obtained in \cite[Ths.~5.1, 5.2]{my-jat}.

\subsection{A characterization of the Besov spaces}

For $1\le p\le \infty$ and $\tau, r>0$, define the Besov space
${B^r_{p,\tau}(\mathbb{R}^d)}$ as the collection of functions $f\in
L^p(\mathbb{R}^d)$ such that
\[
\|f\|_{B^r_{p,\tau}(\mathbb{R}^d)} = \|f\|_{L^p(\mathbb{R}^d)} +
\left(\int_0^1\Bigl(\frac{\Omega_l(f,t)_p}{t^r}\Bigr)^\tau \frac{dt}{t}\right)^{1/\tau}<\infty,
\]
where $0<r<\theta l$. Similarly we define the Lipschitz space
$\mathrm{Lip}^{r}_{p}(\mathbb{R}^d)\equiv {B^r_{p,\infty}(\mathbb{R}^d)}$, i.e.,
\[
\|f\|_{\mathrm{Lip}^{r}_{p}(\mathbb{R}^d)} = \|f\|_{L^p(\mathbb{R}^d)} +
\sup_t \frac{\Omega_l(f,t)_p}{t^r},\qquad 0<r< \theta l.
\]

It turns out that it is possible to characterize functions from the Besov space
${B^r_{p,\tau}(\mathbb{R}^d)}$ in terms of growth properties of their Fourier transforms.

\begin{theorem}\label{besov}
Let $d\ge 1$, $1<\tau\le \infty$, and $\frac{2d}{d+1}<p\le \tau$. If $f\in
\wh{GM}{}^{d}\cap L^{p}(\R^{d})$ and $\wh{f}\ge 0$, then a necessary and
sufficient condition for $f\in B^r_{p,\tau}(\mathbb{R}^d)$ is
\begin{equation}\label{bes}
\int_{0}^{\infty} s^{r\tau+d \tau - d\tau/p-1} F_{0}^\tau(s) \,ds<\infty
\quad\text{if}\quad1<\tau<\infty
\end{equation}
and
\begin{equation}\label{bes1}
\sup_s s^{r+d - d/p} F_{0}(s)<\infty\quad\text{if}\quad\tau=\infty.
\end{equation}
\end{theorem}

\begin{proof}
\textit{The case of $1<\tau < \infty$}.
Let first \eqref{bes} hold.
By \eqref{bes2}, we get
\begin{multline*}
|f|_{B^r_{p,\tau}}\asymp
K_1+K_2+K_3:=
\int_0^1 t^{(\theta l-r)\tau-1}
\left(\int_{0}^{1} s^{\theta l p + dp-d-1} F_{0}^p(s) \,ds\right)^{\tau/p}
dt
\\
{}+
\int_1^\infty t^{(r-\theta l)\tau-1}
\left(\int_{1}^{t} s^{\theta l p + dp-d-1} F_{0}^p(s) \,ds\right)^{\tau/p}
dt
+
\int_0^1 t^{r\tau-1}
\left(\int_{1/t}^\infty s^{dp-d-1} F_{0}^p(s) \,ds\right)^{\tau/p}
dt.
\end{multline*}
Then by H\"{o}lder's inequality with parameters $\alpha=\tau/p$ and $\alpha'$, we get
\[
K_1 \ll
\int_{0}^{1} s^{r\tau+d \tau - d\tau/p-1} F_{0}^\tau(s) \,ds.
\]
By Hardy's inequalities (see e.g. \cite[p.~124]{ben}), we have
\[
K_2+K_3 \ll
\int_{1}^{\infty} s^{r\tau+d \tau - d\tau/p-1} F_{0}^\tau(s) \,ds.
\]
Hence, if \eqref{bes} holds, $f\in B^r_{p,\tau}(\mathbb{R}^d)$.

Let $f\in B^r_{p,\tau}(\mathbb{R}^d)$. By \eqref{bes4},
\[
F_0(s)^\tau \ll
s^{d\tau/p-d\tau}\left(\int_{s/c}^\infty {F_0^p(u)} {u}^{dp-d-1} \,du \right)^{\tau/p}.
\]
Therefore, making use of this, we have
\begin{multline*}
\int_{0}^{\infty} s^{r\tau+d \tau - d\tau/p-1} F_{0}^\tau(s) \,ds
\ll
\int_{1}^{\infty} s^{r\tau-1} \left(\int_{s}^\infty {F_0^p(u)} {u}^{dp-d-1} \,du \right)^{\tau/p} ds
\\
{}+
\int_0^{1} s^{r\tau-1} \left(\int_{s}^\infty {F_0^p(u)} {u}^{dp-d-1} \,du \right)^{\tau/p} ds
\ll
|f|_{B^r_{p,\tau}}
+ \bigl\||\xi|^{d(1-2/p)} \widehat{f}(\xi)\bigr\|_p^\tau \int_0^{1} s^{r\tau-1} ds.
\end{multline*}
Finally, since $\Bigl\||\xi|^{d(1-2/p)} \widehat{f}(\xi)\Bigr\|_p\ll \|f\|_p$
(see \eqref{aseqhff}), \eqref{bes} holds.

\smallbreak
\textit{The case of $\tau=\infty$}.
Let first \eqref{bes1} hold. Then by \eqref{bes2}, \eqref{bes1} yields
\[
\Omega_l^p(f,t)_p
\ll
t^{\theta l p} \int_{0}^{1/t} s^{\theta l p-rp-1} \,ds +
\int^{\infty}_{1/t} s^{-rp-1} \,ds \ll t^{rp},
\]
i.e., $f\in \mathrm{Lip}^{r}_{p}(\mathbb{R}^d)$.

On the other hand, if $f\in \mathrm{Lip}^{r}_{p}(\mathbb{R}^d)$, we use
\eqref{bes4} and \eqref{bes2}
\[
F_0^p(s) \ll
s^{d-dp} \int_{s/c}^\infty {F_0^p(u)} {u}^{dp-d-1} \,du \ll
s^{d-dp} \Omega_l^p(f, 1/s)_p \ll
s^{d-dp-rp},
\]
which is \eqref{bes1}.
\end{proof}

\subsection{Embedding theorems}
The following Sobolev-type embedding result for the Besov space with the limiting smoothness
parameter is well known: $B^r_{p,q}\hookrightarrow L^{q}$,
$r=d\bigl(\frac1p-\frac1q\bigr)$ (see, e.g., \cite[(8.2)]{peetre}). Theorem \ref{besov}
gives the sharpness of this result in the following sense.

\begin{corollary}
Let $d\ge 1$ and $\frac{2d}{d+1}<p< q<\infty$.
If $f\in \wh{GM}{}^{d}\cap L^{p}(\R^{d})$ and $\wh{f}\ge 0$, then
\begin{equation}\label{bes5}
f\in B^r_{p,q}(\mathbb{R}^d), \quad r=d\Bigl(\frac1p-\frac1q\Bigr) \quad\Longleftrightarrow\quad
f\in L^{q}(\mathbb{R}^d).
\end{equation}
\end{corollary}
\begin{proof}To show \eqref{bes5}, we combine Theorem \ref{besov} and
$\bigl\||\xi|^{d(1-2/p)} \widehat{f}(\xi)\bigr\|_p\asymp \|f\|_p$,
$\frac{2d}{d+1}<p<\infty$ (see \eqref{aseqhff}).
\end{proof}
Note that the embedding $B^r_{p,q}\hookrightarrow L^{q}$ is equivalent to the sharp (Ul'yanov) inequalities for moduli of smoothness in different metrics, as recently shown in \cite[Th. 2.4]{trebels}.


\begin{thebibliography}{DeVL}

\bibitem[AS]{AS} M.~Abramowitz, I.~A.~Stegun, eds., \textit{Handbook of Mathematical
Functions with Formulas, Graphs, and Mathematical Tables}, New York: Dover
Publications, 1972.

\bibitem[AW]{AW} R.~Askey, S.~Wainger, \textit{Integrability theorems for
Fourier series}, Duke Math.~J., \textbf{33} (1966), 223--228.

\bibitem[BH]{BH}
J.~J.~Benedetto, H.~P.~Heinig, \textit{Weighted Fourier Inequalities: New
Proofs and Generalizations}, J.~Fourier Anal. Appl., \textbf{9} (2003), 1--37.

\bibitem[BL]{BL} J.~J.~Benedetto, J.~D.~Lakey, \textit{The Definition of the
Fourier Transform for Weighted Inequalities}, J.~Funct. Anal., \textbf{120},
no.~2 (1994), 403--439.

\bibitem[BSh]{ben}
C. Bennett, R. Sharpley, \textit{Interpolation of operators,}
Academic Press, 1988.

\bibitem[BP1]{BP}
W.~Bray, M.~Pinsky, \textit{Growth properties of Fourier transforms via moduli of
continuity}, J.~Funct. Anal., \textbf{255}, no.~9 (2008), 2265--2285.

\bibitem[BP2]{BP2}
W.~Bray, M.~Pinsky, \textit{Growth properties of Fourier transforms},
arXiv:0910.1115

\bibitem[Cl]{cline}
D.~B.~H.~Cline, \textit{Regularly varying rates of decrease for moduli of
continuity and Fourier transforms of functions on $\R^d$}, J. Math. Anal. Appl.
\textbf{159} (1991), 507--519.

\bibitem[DD]{DD}
F.~Dai, Z.~Ditzian, \textit{Combinations of multivariate averages}, J.~Approx.
Theory, \textbf{131}, no.~2 (2004), 268--283.

\bibitem[DDT]{DDT}
F.~Dai, Z.~Ditzian, S.~Tikhonov, \textit{Sharp Jackson inequality}, J.~Approx.
Theory, \textbf{151}, no.~1 (2008), 86--112.

\bibitem[DC]{Ca}
L.~De~Carli, \textit{On the $L^p$--$L^q$ norm of the Hankel transform and
related operators}, J.~Math. Anal. Appl., \textbf{348}, no.~1 (2008), 366--382.

\bibitem[DL]{DeVL}
R.~A.~DeVore, G.~G.~Lorentz, \textit{Constructive approximation}, Berlin:
Springer-Verlag, 1993.

\bibitem[Di]{Di}
Z.~Ditzian, \textit{Smoothness of a function and the growth of its Fourier
transform or its Fourier coefficients}, J.~Approx. Theory, \textbf{162}, no.~5
(2010), 980--986.

\bibitem[DHI]{DIT}
Z.~Ditzian, V.~H.~Hristov, K.~G.~Ivanov, \textit{Moduli of smoothness and
$K$-functionals in $L_p$, $0<p<1$}, Constr. Approx., \textbf{11}, no.~1 (1995),
67--83.

\bibitem[DT]{compt} M.~I.~Dyachenko, S.~Tikhonov,
\textit{Convergence of trigonometric series with general monotone
coefficients}, C.R. Acad. Sci. Paris, Ser. I, \textbf{345} (2007), 123--126.


\bibitem[GK]{GK}
J.~Garc\'{i}a-Cuerva, V.~Kolyada, \textit{Rearrangement estimates for Fourier
transforms in $L^p$ and $H^p$ in terms of moduli of continuity}, Math. Nachr.,
\textbf{228} (2001), 123--144.

\bibitem[Gi]{G}
D.~Gioev, \textit{ Moduli of continuity and average decay of Fourier transforms:
Two-sided estimates}, in: J.~Baik, T.~Kriecherbauer, L.~Li, K.~D.~McLaughlin,
C.~Tomei (Eds.), Integrable Systems and Random Matrices: In Honor of Percy
Deift, in: Contemp. Math., vol.~458, 377–-392, Amer. Math. Soc., 2008.

\bibitem[GLT]{GLT}
D.~Gorbachev, E.~Liflyand, S.~Tikhonov, \textit{Weighted Fourier inequalities:
Boas’ conjecture in $\R^{n}$}, J.~d'Analyse Math., \textbf{114} (2011),
99--120.

\bibitem[HLP]{HLP}
G.~H.~Hardy, J.~E.~Littlewood, G.~P\'olya, \textit{Inequalities}, 2nd~ed.
Cambridge University Press, 1952.

\bibitem[HT]{HT}
E.~Hille, J.~D.~Tamarkin, \textit{On the theory of Fourier transforms}, Bull.
Amer. Math. Soc, \textbf{39} (1933), 768--774.

\bibitem[Ni]{Ni}
S.~M.~Nikolskii, \textit{Approximation of functions of several variables and
imbedding theorems}, Berlin; Heidelberg; New York: Springer, 1975.

\bibitem[Nu]
{nurs}
E.~D.~Nursultanov, \textit{On the coefficients of multiple Fourier series in Lp-spaces}, Izv. RAN Ser. Mat.,
\textbf{64}, (1) (2000), 95--122.

\bibitem[Pe]{peetre}
J.~Peetre, \textit{Espaces d’interpolation et th\'{e}eor\`{e}me de Soboleff}, Ann. Inst.
Fourier (Grenoble), \textbf{16} (1966), 279--317.


\bibitem[SW]{SW}
E.~M.~Stein, G.~Weiss, \textit{Introduction to Fourier analysis on Euclidean
spaces}, Princeton, N.~J., 1971.

\bibitem[Tik]{my-jat}
S.~Tikhonov, \textit{Best approximation and moduli of smoothness: computation
and equivalence theorems}, J.~Approx. Th., \textbf{153} (2008), 19--39.

\bibitem[Tim]{Ti}
M.~F.~Timan, \textit{Best approximation and modulus of smoothness of functions
prescribed on the entire real axis}, Izv. Vyssh. Uchebn. Zaved. Mat., no.~6 (1961),
108--120 (in Russian).

\bibitem[Tit]{Tit}
E.~Titchmarsh, \textit{Introduction to the theory of Fourier integrals}, 2nd~ed.,
Clarendon Press, Oxford University, 1948.


\bibitem[To]{totik}
V.~Totik, \textit{Sharp converse theorem of $L^p$ polynomial approximation}, Constr. Approx., \textbf{4} (1988), 419--433.

\bibitem[Tr1]{Tr}
W.~Trebels, \textit{Estimates for moduli of continuity of functions given by
their Fourier transform}, Lecture Notes in Math., vol.~571, 277--288, Springer,
Berlin, 1977.


\bibitem[Tr2]{trebels}
W.~Trebels, \textit{Inequalities for moduli of smoothness versus embeddings of function spaces},
Arch. Math. 94 (2010), 155--164.

\bibitem[Zy]{Zy}
A.~Zygmund, \textit{Trigonometric series}, vol.~I, II, 3th ed., Cambridge,
2002.

\end{thebibliography}
\end{document}